\documentclass[11pt]{amsart}

\usepackage{geometry}
\usepackage[all]{xy}
\usepackage{epstopdf}
\usepackage{amssymb,amsmath,amsfonts,amsthm,graphicx,enumerate,amscd}
\usepackage{mathrsfs}
\usepackage{ytableau}
\usepackage{lscape}
\usepackage{hyperref}
\usepackage{rotating}

\newcommand{\deff}[1]{\textbf{\emph{\sharp1}}}

\newcommand{\func}[3]{\sharp1 \colon \sharp2 \to \sharp3}
\newcommand{\bb}[1]{\mathbb{\sharp1}}
\newcommand{\lie}[1]{\mathfrak{\sharp1}}
\newcommand{\iprod}[2]{\langle \sharp1, \sharp2 \rangle}
\newcommand{\ddell}[1]{\frac{\partial}{\partial \sharp1}}

\theoremstyle{plain}
\newtheorem{theorem}{Theorem}[section]

\newtheorem{claim}[theorem]{Claim}
\newtheorem{corollary}[theorem]{Corollary}
\newtheorem{lemma}[theorem]{Lemma}
\newtheorem{proposition}[theorem]{Proposition}

\theoremstyle{definition}
\newtheorem{example}[theorem]{Example}
\newtheorem{remark}[theorem]{Remark}

\newtheorem{definition}[theorem]{Definition}

\allowdisplaybreaks

\title{Equivariant formality and representation theory}
\author{Chi-Kwong Fok}

\begin{document}
\maketitle
\begin{abstract}
	Let $G$ be a compact connected Lie group and $K$ its connected Lie subgroup. Using the $K$-theoretic version of equivariant formality developed in \cite{F2} which involves analysis of vector bundles over $G/K$, we find two characterizations of equivariant formality of the isotropy action of $K$ on $G/K$. The first one equates equivariant formality with a ``smoothness" condition of the restriction map of the representation ring of $G$ to that of $K$ and strengthens \cite[Theorem B]{CFok} by rendering the condition of torsion-freeness of $\pi_1(G)$ unnecessary. The second characterization asserts that equivariant formality amounts to the equivariant $K$-theoretic index pairing of two certain special $K$-theory classes being nontrivial in some sense. By applying these characterizations, we are able to give a representation theoretic criterion for equivariant formality of the isotropy action by a circle subgroup, as well as an invariant theory criterion for the case where $K$ is a torus of dimension one less than the rank of $G$, culminating in a complete classification of equivariant formality where $G$ is further assumed to be simple.   
\end{abstract}
\maketitle
\tableofcontents
\section{Introduction}
Equivariant formality, first identified and discussed at length in \cite{BBFMP} and named in \cite{GKM}, is a desirable property of group actions on topological spaces. We say a $G$-action on a space $X$ is equivariantly formal if the Leray-Serre spectral sequence for the rational cohomology of the fiber bundle $X\hookrightarrow X\times_G EG\to BG$ associated with the Borel homotopy quotient collapses on the $E_2$ page. Interesting examples of equivariant formality abound, e.g. Hamiltonian group actions on compact symplectic manifolds and linear algebraic torus actions on smooth complex projective varieties. 

Recently, the study of equivariant formality of compact homogeneous spaces has garnered some interest. Following the convention in \cite{CFok}, we say the pair of compact Lie groups $(G, K)$ is \emph{isotropy formal} if the left action of $K$ on $G/K$ is equivariantly formal, and call such an action the \emph{isotropy action}. When $K$ is the identity subgroup, equivariant formality of isotropy action is immediate. On the other hand, the case when $K$ is of the maximal rank is well-known to be equivariantly formal. More subtle is when $K$ sits between the two extreme cases. The problem of determining isotropy formality in the intermediate cases was taken up in a series of papers. In \cite{ST, Sh}, a partial characterization of isotropy formality was given in terms of formality in the sense of rational homotopy theory and invariant theory of the equivariant cohomology coefficient rings. When $G/K$ is a generalized symmetric space, the isotropy formality of $(G, K)$ was proved in \cite{Go, GN} by comparing cohomological dimensions of the generalized symmetric space and its fixed point set on a case-by-case basis. When $K$ is a circle subgroup, \cite{Ca} provided a complete characterization of isotropy formality. In \cite{CFok}, by exploiting the notion of rational $K$-theoretic equivariant formality (RKEF) developed in \cite{F2}, the authors characterized isotropy formality in terms of the representation rings of some finite covers of $\widetilde{G}$ and $\widetilde{K}$ (such that $\pi_1(\widetilde{G})$ is torsion-free) of $G$ and $K$, and noted that the characterization in fact amounts to ``smoothness" of the embedding of $\widetilde{K}$ into $\widetilde{G}$ in an algebro-geometric sense, if one takes the spec of the representation rings. In fact, considering the representation rings of some finite covers rather than the groups themselves in the characterization is for the sake of convenience as the representation ring of $\widetilde{G}$ is a (Laurent) polynomial ring and has a simpler algebraic structure in general than $R(G)$. As the property of isotropy formality is preserved under taking finite cover (\cite[Theorem 1.2]{Ca}), we expect that there should be a similar characterization where taking finite covers of the Lie groups is not necessary. In this paper we will first prove this strengthened characterization where we remove the condition that $\pi_1(G)$ is torsion-free. 
We will use $R(G; \mathbb{Q})$, $I(G; \mathbb{Q})$ and $K^*_G(X; \mathbb{Q})$ to denote the rationalized representation ring, augmentation ideal and equivariant $K$-theory respectively. 
\begin{theorem}\label{rep}
	Let $G$ be a compact connected Lie group and $K$ its connected Lie subgroup. Then $(G, K)$ is isotropy formal if and only if the image $R:=i^*R(G; \mathbb{Q})$ of the restriction map $i^*: R(G;\mathbb{Q})\to R(K;\mathbb{Q})$ is regular at the restriction  $I:=i^*I(G;\mathbb{Q})$ of the augmentation ideal of $R(G;\mathbb{Q})$.
\end{theorem}
As remarked earlier and in \cite[Introduction]{CFok}, the isotropy formality of $(G, K)$, through the regularity condition in the above characterization, can be recast as the smoothness at the point $I(G; \mathbb{Q})$ of the map $\text{Spec}R(K; \mathbb{Q})\to\text{Spec}R(G; \mathbb{Q})$, which is the affinization of the embedding $K\hookrightarrow G$. As a corollary, the rational representation ring $R(G;\mathbb{Q})$ is always regular at the augmentation ideal, since $(G, G)$ is obviously isotropy formal. In particular, taking the quotient of a compact connected Lie group $G$ by its finite central subgroup $Z$ does not create a ``crease'' at $I(G/Z; \mathbb{Q})$ in $\text{Spec}R(G/Z; \mathbb{Q})$. It was claimed after \cite[Proof of Lemma 4.2]{Ho} that $R(\text{PSU}(3); \mathbb{Q})$ is not regular at its augmentation ideal, citing the existence of torsion in $\pi_1(\text{PSU}(3))$. We will show by explicit computation that $R(\text{PSU}(3); \mathbb{Q})$ is in fact regular at its augmentation ideal, which is consistent with Theorem \ref{rep}. This will serve as a correction to the claim after \cite[Lemma 4.2]{Ho}.

To demonstrate the utility of Theorem \ref{rep}, we apply it to obtain characterizations of isotropy formality in two contexts.  The first one is when $K$ is a circle subgroup of $G$, which is assumed to be semisimple. We characterize isotropy formality for this case in Theorem \ref{circle}. Interestingly, Theorem \ref{rep} reduces the proof of Theorem \ref{circle} to elementary calculus of curves when checking the smoothness of the map $\text{Spec}R(K; \mathbb{Q})\to\text{Spec}R(G; \mathbb{Q})$. In fact this result recovers \cite[Theorem 1.5]{Ca} (in the case where $S$ is contained the commutator subgroup $G'$, i.e., the semisimple part of $G$), where the characterization is in terms of the ``reflectibility'' of $S$ and whose proof involves cohomological arguments.
\begin{theorem}\label{circle}
	Let $S$ be a circular subgroup of a compact connected semisimple Lie group $G$. Then $(G, S)$ is isotropy formal if and only if any representation of $G$, when restricted to $S$, is self-dual. 
\end{theorem}
We also apply Theorem \ref{rep} to the case where $K$ is a torus subgroup of dimension being 1 less than the rank of $G$. Via the Chern character, we translate the regularity criterion for representation rings to the regularity criterion below for a distinguished polynomial on the Lie algebra of $K$ invariant under the Weyl group, which is a computable algorithm for determining whether any given corank 1 pair is isotropy formal.

\begin{theorem}[=Theorem \ref{invtpoly}]
	Let $G$ be a compact connected Lie group of rank $n$, $S$ a torus subgroup of dimension being $n-1$, $T$ a maximal torus of $G$ containing $S$, $\mu\in \chi(T)$ a primitive character of $T$ such that $S=\text{ker }\mu$, $L$ a primitive weight in the weight lattice $\mathfrak{t}^*_\mathbb{Z}$ corresponding to $\mu$ after exponentiation, and $W$ the Weyl group of $G$. Define
	\[\Phi_L:=\left(\prod_{w\in W/W_{[L]}} wL\right)^{n(L)}\in \text{Sym}^*(\mathfrak{t}^*)^W,\]
	where $W_{[L]}$ is the stabilizer of $[L]:=\{\pm L\}\in \mathfrak{t}^*/\{\pm 1\}$, $n(L)=1$ or 2 depending on whether $\displaystyle\prod_{w\in W/W_{[L]}}wL$ is $W$-invariant or $W$-semi-invariant, i.e. there exists a nontrivial one-dimensional real representation $\chi$ of $W$ such that
	\[\displaystyle w'(\prod_{w\in W/W_{[L]}}wL)=\chi(w')\prod_{w\in W/W_{[L]}}wL\text{ for all }w'\in W.\] 
	Then $(G, S)$ is isotropy formal if and only if $\Phi_L\notin (f_1, \cdots, f_n)^2$, where $f_1, \cdots, f_n$ are the fundamental $W$-invariant homogeneous polynomials on $\mathfrak{t}$, i.e. $\text{Sym}^*(\mathfrak{t}^*)^W=\mathbb{R}[f_1, \cdots, f_n]$.
\end{theorem}
Theorem \ref{invtpoly}, together with a sharp bound on the degree on the distinguished polynomial (Lemma \ref{degreebound}), leads us to the following classification of isotropy formal action of corank 1 for $G$ being a compact and connected simple Lie group.
\begin{theorem}\label{corank1}
		Let $G$ be a compact and connected simple Lie group of rank $n$, $S$ a torus subgroup of dimension being $n-1$, $T$ a maximal torus of $G$ containing $S$, $\mu\in \chi(T)$ a primitive character of $T$ such that $S=\text{ker }\mu$, $L$ a primitive weight in the weight lattice $\mathfrak{t}^*_\mathbb{Z}$ corresponding to $\mu$ after exponentiation, and $W$ the Weyl group of $G$. The following table classifies the isotropy formality of $(G, S)$ in terms of $L$. Here $\omega_i$ is the $i$-th fundamental weight of the relevant simple Lie group following the Bourbaki convention.
		\begin{center}
		\begin{tabular}{|c|c|}
		\hline
		$G$&$L$\\
		\hline
		$A_n, n\geq 2, n\neq 3$& $L\in W\cdot \omega_1$\\
		\hline
		$A_3$&$L\in W\cdot \omega_1\cup W\cdot \omega_2$\\
		\hline
		$B_2, C_2$&$\text{Any primitive weight}$\\
		\hline
		$B_n, C_n, n=3, 4$&$L\in W\cdot \omega_1\cup W\cdot \omega_n$\\
		\hline
		$B_n, C_n, n\geq 5$&$L\in W\cdot \omega_1$\\
		\hline
		$D_4$&$L\in W\cdot\omega_1\cup W\cdot \omega_3\cup W\cdot \omega_4$\\
		\hline
		$D_n, n\geq 5$&$L\in W\cdot \omega_1$\\
		\hline
		$G_2$&$\text{Any primitive weight}$\\
		\hline
		$F_4, E_6, E_7, E_8$&$\text{No such primitive weights}$\\
		\hline
		\end{tabular}
		\end{center}
\end{theorem}
It should be noted that a complete characterization of isotropy formality of general corank 1 pairs $(G, K)$ is given in \cite{CH}. This characterization consists of an elaborate algorithm of determining isotropy formality based on the fundamental groups and total Betti number of a certain associated homogeneous space arising from a pair of compact Lie groups dubbed the maximal regular pair. When the rational homotopy type of such an associated homogeneous space is a product of an odd- and an even-dimensional spheres, its isotropy formality can be determined using Table 1 of \cite{CH}, which classifies irreducible compact connected pairs with the associated homogeneous spaces simply-connected and its rational cohomology isomorphic to that of the product of two spheres of dimensions of distinct parities, and determines their isotropy formality. The isotropy formal pairs in the table are consistent with those listed in Theorem \ref{corank1}. Arguably our characterization of isotropy formal corank 1 pairs (Theorem \ref{invtpoly}) is more direct in the sense that it only requires studying a certain invariant polynomial which is readily computable and does not depend on any classification results about homogeneous spaces other than the classification of simple Lie groups. Moreover, our characterization enables us to give a uniform treatment to the classification of isotropy formality of corank 1 pairs where $G$ is a connected simple Lie group, and show explicitly that when $G$ is an exceptional Lie group but not $G_2$, then the pair is not isotropy formal for any corank 1 torus subgroups. 

We also find another characterization of isotropy formality with a different flavor. We first observe that formality of $G/S$ (where $S$ is a torus subgroup) is equivalent to the nontrivial $K$-theoretic index pairing of certain two special $K$-theory classes (Proposition \ref{nondegenformal}). Based on this we obtain the following equivariant analogue. 
\begin{theorem}\label{indexthry}
	Let $G$ be a compact connected Lie group and $S$ a torus subgroup, and $\pi: G/S\to\text{pt}$ be the collapsing map. Denote by $A$ the subring of $K_S^*(G/S)\otimes\mathbb{Q}$ generated by $S$-equivariant line bundles $G\times_S\mathbb{C}_\mu$ where $\mu\in\mathfrak{s}^*$. Define $\mathcal{L}_S$ to be the $R(S; \mathbb{Q})$-submodule of $K_S^*(G/S)\otimes\mathbb{Q}$ generated by the odd equivariant $K$-theory classes of $G/S$ defined by the $\delta$-construction (see Definitions \ref{deltaconstruction} and \ref{equivsamelson}). Then $(G, S)$ is isotropy formal if and only if there exist $a\in A$ and $b\in\bigwedge\nolimits_{R(S; \mathbb{Q})}^{\text{rank }G-\text{rank }S}\mathcal{L}_S$ such that $\pi_*(ab)\notin I(S; \mathbb{Q})$. 
\end{theorem}
\cite[Theorem A]{CFok} asserts that isotropy formality of the pair $(G, K)$ implies the formality of $G/K$. Theorem \ref{indexthry} and Proposition \ref{nondegenformal} can be thought of as a further testament to the slogan ``isotropy formality is the equivariant analogue of formality of compact homogeneous spaces", in the sense that isotropy formality amounts to the nontriviality of a certain equivariant $K$-theoretic index pairing, which implies the nontriviality of a certain ordinary $K$-theoretic index pairing, a condition equivalent to the formality of compact homogeneous spaces. It should be noted, however, that equivariant formality in general does not imply formality. For example, the trivial $G$-action on a non-formal manifold is always equivariantly formal.

We end this paper by discussing the isotropy formality of two examples of pairs using the aforementioned main results.

Throughout this paper, $G$ always means a compact connected Lie group and $K$ its connected Lie subgroup unless otherwise specified. 

\textbf{Acknowledgment }We would like to gratefully acknowledge the support by the Xi'an Jiaotong-Liverpool University Research Development Fund RDF 23-02-029.
\section{Preliminaries}
In this section, we will recall a number of definitions and results about the (equivariant) $K$-theory of compact homogeneous spaces and $K$-theoretic equivariant formality, which are all crucial ingredients in the proofs of the main results of this paper.
\begin{definition}(\cite[\S I.4]{Ho}, \cite[\S 3]{BZ}, \cite[Definitions 2.2 and 2.5]{F})\label{deltaconstruction}
	Let $\delta: R(G)\to K^{-1}(G)$ be the $\mathbb{Z}$-linear map which sends a complex representation $\rho$ with underlying vector space $V$ to the $K$-theory class represented by the complex of vector bundles
	\begin{align*}
		0\longrightarrow G\times\mathbb{R}\times V&\longrightarrow G\times\mathbb{R}\times V\longrightarrow 0\\
		(g, t, v)&\mapsto\begin{cases}(g, t, -t\rho(g)v),&\ \text{if }t>0\\ (g, t, tv),&\ \text{if }t\leq 0.\end{cases}
	\end{align*}
	More generally, for an element $\rho\in\text{ker}(i^*: R(G)\to R(K))$ which is the formal difference $\rho_1-\rho_2\in R(G)$ of two complex $G$-representations restricting to the same $K$-representations, we define the $\mathbb{Z}$-linear map $\delta^{G/K}:\text{ker}(i^*)\to K^{-1}(G/K)$ which sends $\rho$ to the complex of vector bundles
	\begin{align*}
		0\longrightarrow G/K\times\mathbb{R}\times V&\longrightarrow G/K\times\mathbb{R}\times V\longrightarrow 0\\
		(gK, t, v)&\mapsto\begin{cases}(gK, t, t\rho_1(g)\rho_2(g)^{-1}v),&\ \text{if }t>0\\ (gK, t, tv),&\ \text{if }t\leq 0.\end{cases}
	\end{align*}
	The map $\delta_K^{G/K}: \text{ker}(i^*)\to K_K^{-1}(G/K)$ is defined using the same complex of vector bundles and equipping it with the $K$-action on $G/K\times \mathbb{R}\times V$ given by $k\cdot(gK, t, v)=(kgK, t, \rho_1(k)v)$. The map $\delta_K^G: R(G)\to K_K^{-1}(G)$ can be defined similarly (here $K$ acts on $G$ by conjugation). We will also use the same notation to denote the rationalized maps from the rational representation ring to the rational $K$-theory $K^{-1}(G; \mathbb{Q})$, $K^{-1}(G/K; \mathbb{Q})$ etc.
\end{definition}
\begin{definition}(\cite[p. 172]{H}, \cite{GD}). Let $A$ be a ring and $B$ an $A$-algebra. The module of K\"ahler differentials $\Omega^1_{B/A}$ of $B$ over $A$ is the quotient of the $B$-module freely generated by the symbols $\{db| b\in B\}$ by the submodule generated by the following sets
	\begin{enumerate}
		\item(Constancy of elements of $A$) $\{da| a\in A\}$, 
		\item(Linearity of $d$) $\{d(b_1+b_2)-db_1-db_2| b_1, b_2\in B\}$, and 
		\item(Derivation) $\{d(b_1b_2)-b_1db_2-b_2db_1| b_1, b_2\in B\}$.
	\end{enumerate}
	The Grothendieck differentials $\Omega^*_{B/A}$ is defined to be the exterior $B$-algebra $\displaystyle\bigoplus_{p=0}^\infty \bigwedge\nolimits_B^p\Omega_{B/A}^1$.
\end{definition}
\begin{theorem}\label{kthrygrp}
	\begin{enumerate}
		\item \textnormal{(\cite[Theorem 2.1]{Ho}).}\label{Ho} The map $\delta$ satisfies
		\[\delta(\rho_1\otimes\rho_2)=\text{dim}(\rho_1)\delta(\rho_2)+\text{dim}(\rho_2)\delta(\rho_1).\]
		Moreover, $\delta$ induces an isomorphism of Hopf algebras
		\[\bigwedge\nolimits_\mathbb{Q}^*(\text{Im}(\delta)\otimes\mathbb{Q})\cong K^*(G; \mathbb{Q}).\]
		\item \textnormal{(cf. \cite{BZ})}\label{BZ} The map 
		\begin{align*}
			\Omega^*_{R(G)/\mathbb{Z}}\otimes_{R(G)}R(K)&\to K_K^*(G)\\
			d\rho_1\otimes\rho_2&\mapsto \rho_2\cdot\delta_K^G(\rho_1)
		\end{align*}
		is injective. In particular, when $\pi_1(G)$ is torsion-free, the map is an isomorphism.
	\end{enumerate}
\end{theorem}
\begin{remark}\label{deltaredux}
	We have that $\Omega^1_{R(G; \mathbb{Q})/\mathbb{Q}}\otimes_{R(G; \mathbb{Q})}\mathbb{Q}\cong\text{Im}(\delta)\otimes\mathbb{Q}$ via the map $d\rho\otimes 1\mapsto \delta(\rho)$, and that $\delta^{G/K}$ also satisfies the derivation property of $\delta$ as in Theorem \ref{kthrygrp} (\ref{Ho}).
\end{remark}
\begin{proposition}\label{log}
	Let $\rho_1\cdot\rho_2: G\to \text{GL}(V)$ be the pointwise multiplication $(\rho_1\cdot\rho_2)(g):=\rho_1(g)\rho_2(g)$. Then $\delta(\rho_1\cdot\rho_2)=\delta(\rho_1)+\delta(\rho_2)$.
\end{proposition}
\begin{proof}
	This is an immediate consequence of the fact that the two maps $G \to \text{GL}(V\oplus V)$ defined by 
	\[g\mapsto \begin{pmatrix}\rho_1(g)& 0\\ 0&\rho_2(g)\end{pmatrix}\ \text{and }g\mapsto \begin{pmatrix}(\rho_1\cdot\rho_2)(g)&0\\ 0& \text{I}_V\end{pmatrix}\]
	are $G$-equivariantly homotopic (cf. \cite[Proof of Proposition 3.1]{BZ}) and that $\delta(n)=0$ for $n\in\mathbb{Z}$.
\end{proof}
\begin{proposition}\label{repringreg}
	$\text{dim}_\mathbb{Q}\Omega^1_{R(G; \mathbb{Q})/\mathbb{Q}}\otimes_{R(G; \mathbb{Q})}\mathbb{Q}=\text{rank }G$
\end{proposition}
\begin{proof}
	The LHS is isomorphic to $\text{Im}(\delta)\otimes\mathbb{Q}$, which is the primitive space of $K^*(G; \mathbb{Q})$ by Theorem \ref{kthrygrp} (\ref{Ho}). Note that $\text{dim }K^*(G; \mathbb{Q})=\text{dim }H^*(G; \mathbb{Q})=2^{\text{rank }G}$ (for the second equality, see, for example, \cite[Theorem 3.6]{F3}), so the dimension of the primitive space is $\text{rank }G$. 
\end{proof}
\begin{proposition}\label{rkkahler}
	$\text{rank}_{R(K; \mathbb{Q})}\Omega^1_{R(G; \mathbb{Q})/\mathbb{Q}}\otimes_{R(G; \mathbb{Q})}R(K; \mathbb{Q})=\text{rank }G$.
\end{proposition}
\begin{proof}
	Let $R:=i^*R(G; \mathbb{Q})$. It suffices to show that $\text{rank}_R\Omega^1_{R(G; \mathbb{Q})/\mathbb{Q}}\otimes_{R(G; \mathbb{Q})}R=\text{rank}G$ because $R$ is a subring of $R(K; \mathbb{Q})$. As $R(K; \mathbb{Q})$ is a subring of the (rationalized) representation ring of a maximal torus of $K$, which is an integral domain, $R$ is also an integral domain. Then $\text{ker}(i^*)$ is a prime ideal of $R(G; \mathbb{Q})$. Moreover, $\text{ker}(i^*)\subseteq I(G; \mathbb{Q})$. For simplicity we denote $\text{ker}(i^*)$ by $\mathfrak{p}$. By \cite[Lemmas 7.12(2) and 7.13]{CFok}, we have
	\[\text{rank}_{R(G; \mathbb{Q})}\Omega^1_{R(G; \mathbb{Q})/\mathbb{Q}}=\text{rank}_{R(G; \mathbb{Q})_{\mathfrak{p}}}\Omega^1_{R(G; \mathbb{Q})/\mathbb{Q}}\otimes_{R(G; \mathbb{Q})}R(G; \mathbb{Q})_{\mathfrak{p}}=\text{rank}G.\]
	It follows that 
	\begin{align*}
		&\text{rank}G\\
		=&\text{rank}_{R(G; \mathbb{Q})_{\mathfrak{p}}}\Omega^1_{R(G; \mathbb{Q})/\mathbb{Q}}\otimes_{R(G; \mathbb{Q})}R(G; \mathbb{Q})_{\mathfrak{p}}\\
		\leq &\text{rank}_{R(G; \mathbb{Q})/\mathfrak{p}}\Omega^1_{R(G; \mathbb{Q})/\mathbb{Q}}\otimes_{R(G; \mathbb{Q})}(R(G; \mathbb{Q})_{\mathfrak{p}}/\mathfrak{p}R(G; \mathbb{Q})_{\mathfrak{p}})\ (\text{By Nakayama's Lemma})\\
		=&\text{rank}_{R(G; \mathbb{Q})/\mathfrak{p}}\Omega^1_{R(G; \mathbb{Q})/\mathbb{Q}}\otimes_{R(G; \mathbb{Q})}(R(G; \mathbb{Q})_{I(G; \mathbb{Q})}/\mathfrak{p}R(G; \mathbb{Q})_{I(G; \mathbb{Q})})\\
		&(\text{As }R(G; \mathbb{Q})_{\mathfrak{p}}/\mathfrak{p}R(G; \mathbb{Q})_{\mathfrak{p}}\cong R(G; \mathbb{Q})_{I(G; \mathbb{Q})}/\mathfrak{p}R(G; \mathbb{Q})_{I(G; \mathbb{Q})})\\
		\leq&\text{rank}_{R(G; \mathbb{Q})/I(G; \mathbb{Q})}\Omega^1_{R(G; \mathbb{Q})/\mathbb{Q}}\otimes_{R(G; \mathbb{Q})}(R(G; \mathbb{Q})_{I(G; \mathbb{Q})}/I(G; \mathbb{Q})R(G; \mathbb{Q})_{I(G; \mathbb{Q})})\\
		&(\text{By Nakayama's Lemma})\\
		=&\text{rank}_{\mathbb{Q}}\Omega^1_{R(G; \mathbb{Q})/\mathbb{Q}}\otimes_{R(G; \mathbb{Q})}\mathbb{Q}\\
		=&\text{rank}G\ (\text{By Proposition \ref{repringreg}}).
	\end{align*}
	Thus 
	\begin{align*}
		\text{rank}_R\Omega^1_{R(G; \mathbb{Q})/\mathbb{Q}}\otimes_{R(G; \mathbb{Q})}R&=\text{rank}_{R(G; \mathbb{Q})/\mathfrak{p}}\Omega^1_{R(G; \mathbb{Q})/\mathbb{Q}}\otimes_{R(G; \mathbb{Q})}(R(G; \mathbb{Q})_{\mathfrak{p}}/\mathfrak{p}R(G; \mathbb{Q})_{\mathfrak{p}})\\
		&=\text{rank}G
	\end{align*}
	as desired.
\end{proof}
\begin{definition}(cf. \cite[Definition 6.1]{CFok}).\label{equivsamelson}
	Let $Q$ be the $K$-theoretic Samelson space, i.e., the space of primitive elements of $K^*(G; \mathbb{Q})$ in the image of the pullback map $j^*: K^*(G/K; \mathbb{Q})\to K^*(G; \mathbb{Q})$ induced by the quotient map $j: G\to G/K$. Similarly, let $Q_K$ be the equivariant $K$-theoretic Samelson space, i.e., $\text{Im}(j_K^*)\cap \mathcal{N}_K$, where $j_K^*: K_K^*(G/K; \mathbb{Q})\to K_K^*(G; \mathbb{Q})$ is the pullback map on equivariant $K$-theory, and $\mathcal{N}_K$ the $\mathbb{Q}$-vector subspace of $K_K^*(G; \mathbb{Q})$ spanned by the primitive elements $\{\delta_K^G(\rho)|\rho\in R(G; \mathbb{Q})\}$. Let $\mathcal{M}_K$ be the $\mathbb{Q}$-vector subspace of $K_K^*(G; \mathbb{Q})$ generated by $\{\delta_K^G(\rho)|\rho\in\text{ker}(i^*)\}$, and $\mathcal{L}_K$ the $R(K; \mathbb{Q})$-submodule of $K_K^*(G/K; \mathbb{Q})$ generated by $\{\delta_K^{G/K}(\rho)|\rho\in\text{ker}(i^*)\}$.
\end{definition}
The following result about the structure of the $K$-theory of $G/K$ is deduced from the structure of its cohomology (\cite[pp. 73, 83, 152]{GHV}) and by applying the Chern character isomorphism.
\begin{proposition}\textnormal{(cf. \cite[Theorem 6.2]{CFok}).}\label{kthrystructure}
	We have the ring isomorphism
	\[K^*(G/K; \mathbb{Q})\cong (R(K; \mathbb{Q})/i^*I(G; \mathbb{Q})\oplus\mathfrak{a})\otimes\bigwedge\nolimits^*P.\]
	Here $R(K; \mathbb{Q})/i^*(G; \mathbb{Q})$ is the image of the map $\alpha: R(K; \mathbb{Q})\to K^0(G/K; \mathbb{Q})$ defined by the associated vector bundle $\alpha(\rho):=[G\times_K V_\rho]$, $\mathfrak{a}$ is an ideal of the first tensor factor on the RHS. We also have that $j^*$ maps $P$ isomorphically onto $Q$. Moreover, $\text{dim }P=\text{dim }Q\leq \text{rank }G-\text{rank }K$, and equality holds if and only if $G/K$ is formal, if and only if $\mathfrak{a}=0$.
\end{proposition}
\begin{lemma}\label{key}
	We have $\delta^{G/K}(\text{ker}(i^*))\subseteq P$.
\end{lemma}
\begin{proof}
	Let $x\in\text{ker}(i^*)$ be the formal difference $\rho_1-\rho_2$ of complex $G$-representations restricting to the same $K$-representation. Then $j^*(\delta^{G/K}(x))=\delta(\rho_1\cdot\rho_2^{-1})=\delta(\rho_1)-\delta(\rho_2)\in Q$ (the last equality holds by Proposition \ref{log}) and thus $\delta^{G/K}(x)\in P$. It follows that $\delta^{G/K}(\text{ker}(i^*))\subseteq P$.
\end{proof}
A main ingredient in our representation theoretic proof of Theorem \ref{rep} is the following $K$-theoretic characterization of equivariant formality.
\begin{theorem}\textnormal{(\cite[Theorem 1.3]{F2}).}\label{RKEF} Let $G$ act on a finite CW-complex $X$. Then the following are equivalent. 
\begin{enumerate}
	\item $X$ is an equivariantly formal $G$-space.
	\item The forgetful map $f: K_G^*(X; \mathbb{Q})\to K^*(X; \mathbb{Q})$ is surjective. 
	\item The forgetful map $f$ induces an isomorphism $K_G^*(X; \mathbb{Q})\otimes_{R(G; \mathbb{Q})}\mathbb{Q}\to K^*(X; \mathbb{Q})$.
\end{enumerate}
\end{theorem}
Later on we will, by abuse of notation, always use $f$ to denote the forgetful map for any topological spaces $X$. 
\section{Proof of Theorems \ref{rep} and \ref{circle}} 
\begin{lemma}\label{eqSam}
	The equivariant $K$-theoretic Samelson subspace $Q_K$ is the $\mathbb{Q}$-vector subspace $\mathcal{M}_K$ of $K_K^*(G; \mathbb{Q})$.
\end{lemma}
\begin{proof}
	Let us first prove the case where $K$ is a ($k$-dimensional) torus subgroup $S$. The vector subspace $\mathcal{M}_S$ is in $Q_S$ because $j_S^*(\delta_K^{G/K}(\rho))=\delta_K^G(\rho)$ if $\rho\in\text{ker}(i^*)$. The remaining task is to show that $Q_S\subseteq\mathcal{M}_S$. Let $r_S:S\to G$ and $r_{S/S}: S/S\to G/S$ be inclusion maps, and $k_S:S\to S/S$ be the quotient map. Pick an arbitrary element in $Q_S$. Multiplying an integer if necessary, we may assume that the element is of the form $\delta_S^G(\rho_1-\rho_2)$ where $\rho_1$ and $\rho_2$ are representations of $G$. Let $x\in K^{-1}_S(G/S; \mathbb{Q})$ be a preimage of $\delta_S^G(\rho_1-\rho_2)$ under $j_S^*$. Then we have 
	\begin{align*}
		r_S^*(\delta_S^G(\rho_1-\rho_2))&=r_S^*(j_S^*(x))\\
								&=k_S^*(r_{S/S}^*(x))\ (\text{because }j_S\circ r_S=r_{S/S}\circ k_S)\\
								&=0\ (\text{because }r_{S/S}^*(x)\in K_S^{-1}(S/S; \mathbb{Q})=0).
	\end{align*}
	On the other hand, suppose $\rho_1$ and $\rho_2$ decomposes into the direct sums of irreducible 1-dimensional representations of $S$, $\displaystyle\alpha:=\sum_m t_1^{\alpha_{1m}}t_2^{\alpha_{2m}}\cdots t_k^{\alpha_{km}}$ and $\displaystyle \beta:=\sum_n t_1^{\beta_{1n}}t_2^{\beta_{2n}}\cdots t_k^{\beta_{kn}}$ on restriction to $S$. Then by a generalization of the description of the restriction map $K_T^*(G)\to K_T^*(T)$ with $T$ being a maximal torus of $G$ (cf. \cite[Lemma 4.19]{F}), we have
	\begin{align*}
		&r_S^*(\delta_S^G(\rho_1-\rho_2))\\
		=&\sum_m t_1^{\alpha_{1m}}t_2^{\alpha_{2m}}\cdots t_k^{\alpha_{km}}\otimes\left(\sum_{i=1}^k\alpha_{im}\delta(t_i)\right)-\sum_n t_1^{\beta_{1n}}t_2^{\beta_{2n}}\cdots t_k^{\beta_{kn}}\otimes\left(\sum_{i=1}^n\beta_{in}\delta(t_i)\right)\\
		=&\sum_{i=1}^k t_i\left(\frac{\partial \alpha}{\partial t_i}-\frac{\partial\beta}{\partial t_i}\right)\otimes\delta(t_i)
	\end{align*}
	Thus $\displaystyle\frac{\partial}{\partial t_i}(\alpha-\beta)=0$ for all $1\leq i\leq k$, and $\alpha$ and $\beta$ differ by a trivial representation of dimension say, $\ell$. Then $\delta_S^G(\rho_1-\rho_2)=\delta_S^G(\rho_1-\rho_2-\ell)$ and $\rho_1-\rho_2-\ell\in\text{ker}(i^*)$. This shows $Q_S\subseteq \mathcal{M}_S$.

	As to the general case, we let $S$ be a maximal torus of $K$ and consider the following commutative diagram.
	\begin{eqnarray}
		\xymatrix{K_K^*(G/K; \mathbb{Q})\ar[r]^{j_K^*}\ar[d]_{f_1^*}& K_K^*(G; \mathbb{Q})\ar[d]^{f_2^*}\\ K_S^*(G/S; \mathbb{Q})\ar[r]^{j_S^*}&K_S^*(G; \mathbb{Q})}
	\end{eqnarray}
	Here $f_1^*$ is the composition of the map $K_K^*(G/K; \mathbb{Q})\to K_S^*(G/K; \mathbb{Q})$ induced by $R(K; \mathbb{Q})\to R(S; \mathbb{Q})$ and the map $K_S^*(G/K; \mathbb{Q})\to K_S^*(G/S; \mathbb{Q})$ induced by the projection $G/S\to G/K$, while $f_2^*$ is induced by $R(K; \mathbb{Q})\to R(S; \mathbb{Q})$. Again it is enough to consider $\delta_K^G(\rho_1-\rho_2)\in Q_K$, where $\rho_1$ and $\rho_2$ are $G$-representations, and show that $\rho_1|_K$ and $\rho_2|_K$ differ by a trivial representation. Note that $f_2^*(\delta_K^G(\rho_1-\rho_2))=\delta_S^G(\rho_1-\rho_2)\in Q_S=\mathcal{M}_S$. From the proof of the previous case, we have that $\rho_1|_S-\rho_2|_S$ is a trivial representation of $S$. Since the map $R(K; \mathbb{Q})\to R(S; \mathbb{Q})$ is injective, $\rho_1|_K-\rho_2|_K$ is a trivial representation of $K$. This shows $Q_K\subseteq \mathcal{M}_K$. Together with the easier inclusion $\mathcal{M}_K\subseteq Q_K$, we complete the proof of the proposition. 
\end{proof}
\begin{corollary}\label{pequaldel}
		If $(G, K)$ is isotropy formal, then $P=\delta^{G/K}(\text{ker}(i^*))$.
\end{corollary}
\begin{proof}
	Consider the following commutative diagram
	\begin{eqnarray}
		\xymatrix{K_K^*(G/K; \mathbb{Q})\ar[r]^{j_K^*}\ar[d]^f&K_K^*(G; \mathbb{Q})\ar[d]^f\\ K^*(G/K; \mathbb{Q})\ar[r]^{j^*}& K^*(G; \mathbb{Q})}
	\end{eqnarray}
	where both the vertical maps are forgetful maps. By Theorem \ref{RKEF}, the left vertical map is surjective. Together with the commutativity of the diagram, we have $f(Q_K)=Q$. By Lemma \ref{eqSam}, $f(Q_K)=f(\mathcal{M}_K)$, which is $\delta(\text{ker}(i^*))$. As $j^*$ maps $P$ onto $Q$ isomorphically by Proposition \ref{kthrystructure}, we have $P=\delta^{G/K}(\text{ker}(i^*))$ as desired.
\end{proof}
\begin{remark}
The converse of Corollary \ref{pequaldel} is not true. Consider the pair $(SU(6), SU(3)\times SU(3))$ where $SU(3)\times SU(3)$ is embedded into $SU(6)$ as block diagonal matrices. This pair is known to be not formal (cf. \cite[pp. 486--488]{GHV}). So by Proposition \ref{kthrystructure}, $\text{dim }P<\text{rank }SU(6)-\text{rank }SU(3)\times SU(3)=1$ and hence $P=0$. By Lemma \ref{key}, $\delta^{G/K}(\text{ker}(i^*))=0$ as well. However, the pair is not isotropy formal (cf. \cite[Example 3.5]{CFok}).
\end{remark}
\begin{lemma}
	The $R(K; \mathbb{Q})$-module $\mathcal{L}_K$ is of rank $\text{rank}G-\text{rank}K$.
\end{lemma}
\begin{proof}
	Consider the conormal bundle sequence 
		\begin{eqnarray}\label{conormseq}\text{ker}(i^*)/(\text{ker}(i^*))^2\longrightarrow \Omega^1_{R(G; \mathbb{Q})}\otimes_{R(G; \mathbb{Q})}R\longrightarrow \Omega^1_{R/\mathbb{Q}}\longrightarrow 0.\end{eqnarray}
	Tensoring with $R(S; \mathbb{Q})$ over $R$, we have
	\begin{eqnarray*}\text{ker}(i^*)/\text{ker}(i^*)^2\otimes_R R(K; \mathbb{Q})\longrightarrow\Omega^1_{R(G; \mathbb{Q})/\mathbb{Q}}\otimes_{R(G; \mathbb{Q})}R(K; \mathbb{Q})\longrightarrow\Omega^1_{R/\mathbb{Q}}\otimes_R R(K; \mathbb{Q})\longrightarrow 0.\end{eqnarray*}
	The image of the first map, which is the $R(K; \mathbb{Q})$-module generated by $\{\delta_K^G(\rho)|\rho\in\text{ker}(i^*)\}$, is isomorphic to $\mathcal{L}_K$. 
	Localizing further to the fraction field of $R(K; \mathbb{Q})$ and invoking exactness, we have
	\[\text{rank}\mathcal{L}_K=\text{rank}_{R(K; \mathbb{Q})}\Omega^1_{R(G; \mathbb{Q})/\mathbb{Q}}\otimes_{R(G; \mathbb{Q})}R(K; \mathbb{Q})-\text{rank}_{R(K; \mathbb{Q})}\Omega^1_{R/\mathbb{Q}}\otimes_RR(K; \mathbb{Q}).\]
	Here $\text{rank}\Omega^1_{R(G; \mathbb{Q})/\mathbb{Q}}\otimes_{R(G; \mathbb{Q})}R(K; \mathbb{Q})=\text{rank}G$ by Proposition \ref{rkkahler}, and $\text{rank}\Omega^1_{R/\mathbb{Q}}\otimes_RR(K; \mathbb{Q})=\text{rank}_R\Omega^1_{R/\mathbb{Q}}=\text{Krull dim} R=\text{Krull dim} R(K; \mathbb{Q})=\text{rank}K$, with the first equality due to $R$ being a subring of $R(K; \mathbb{Q})$, the second one due to \cite[Lemma 7.13]{CFok} and the last one due to the going-up theorem and $R(K; \mathbb{Q})$ being integral over $R$ (cf. \cite[Remark, p.60]{AM} and \cite[Proposition 3.2]{Se}). All in all, $\mathcal{L}_K$ is of rank $\text{rank}G-\text{rank}K$
\end{proof}
	\begin{proposition}\label{rankSam}
		The pair $(G, K)$ is isotropy formal if and only if $\text{dim }\delta^{G/K}(\text{ker}(i^*))=\text{rank }G-\text{rank }K$.
	\end{proposition}
	\begin{proof}Suppose $(G, K)$ is isotropy formal. Then by \cite[Theorem A]{CFok}, $G/K$ is formal and by Proposition \ref{kthrystructure}, $\text{dim }P=\text{rank }G-\text{rank }K$. By Corollary \ref{pequaldel}, $\text{dim }\delta^{G/K}(\text{ker}(i^*))=\text{rank }G-\text{rank }K$. Conversely, suppose that $\text{dim }\delta^{G/K}(\text{ker}(i^*))=\text{rank }G-\text{rank }K$. Then by Proposition \ref{kthrystructure} and Lemma \ref{key}, 
	\[\text{rank }G-\text{rank }K=\text{dim }\delta^{G/K}(\text{ker}(i^*))\leq \text{dim }P\leq \text{rank }G-\text{rank }K.\]
	Thus $P=\delta^{G/K}(\text{ker}(i^*))$ and $\text{dim }P=\text{rank }G-\text{rank }K$. By Proposition \ref{kthrystructure} again, we have the isomorphism $K^*(G/K; \mathbb{Q})\cong \left(R(K; \mathbb{Q})/i^*I(G; \mathbb{Q})\right)\otimes\bigwedge\nolimits^*P$. The first tensor factor consists of $K$-theory classes represented by associated vector bundles $G\times_K V$, which admits the equivariant structure via the left $K$-action on $G$. The second tensor factor admits equivariant lifts as well by the fact that the $K$-theory forgetful map $K^*_K(G/K; \mathbb{Q})\to K^*(G/K; \mathbb{Q})$ maps $\delta^{G/K}_K(\text{ker}(i^*))$ onto $P$. We have that the forgetful map is onto and $(G, K)$ is indeed isotropy formal by Theorem \ref{RKEF}. 
	\end{proof}
	\begin{proof}[Proof of Theorem \ref{rep}]
	By Proposition \ref{rankSam}, it remains to show that $\text{dim }\delta^{G/K}(\text{ker}(i^*))=\text{rank }G-\text{rank }K$ if and only if $R:=i^*R(G; \mathbb{Q})$ is regular at the ideal $I:=i^*I(G; \mathbb{Q})$. By \cite[Corollary 3.7]{ABC+}, $R$ is regular at $I$ if and only if 
	\[\text{dim}_\mathbb{Q}(R/I)\otimes_R \Omega^1_{R/\mathbb{Q}}=\text{Krull dim }R_I, \text{where }R/I\cong\mathbb{Q}.\]
	Since $R$ is an integral domain finitely generated over $\mathbb{Q}$ and $I$ is a maximal ideal, the Krull dimension of $R_I$ equals the Krull dimension of $R$ (cf. \cite[Corollary 11.27]{AM}), which in turn is also the Krull dimension of $R(K; \mathbb{Q})$ since $R(K; \mathbb{Q})$ is integral over $R$ (cf. \cite[Proposition 3.2]{Se}, \cite[Remark, p.60]{AM}). By \cite[Lemma 7.12]{CFok}, the Krull dimension of $R(K)$ is rank $K$. Now we have shown another
	\begin{claim}
		The ring $R$ is regular at the ideal $I$ if and only if $\text{dim}_\mathbb{Q}(R/I)\otimes_R\Omega^1_{R/\mathbb{Q}}=\text{rank }K$. 
	\end{claim}
	To finish the proof, we shall show that 
	\begin{claim}
		$\text{dim}_\mathbb{Q}(R/I)\otimes_R\Omega^1_{R/\mathbb{Q}}=\text{rank }K$ if and only if $\text{dim }\delta(\text{ker}(i^*))=\text{rank }G-\text{rank }K$.
	\end{claim} 
	Note that in the conormal bundle sequence (\ref{conormseq}), the first map sends $\rho+(\text{ker}(i^*))^2$ to $d\rho\otimes 1$. Tensoring the sequence with $\cdot\otimes_R R/I$ yields the following sequence of $\mathbb{Q}$-vector spaces.
	\[(\text{ker}(i^*)/(\text{ker}(i^*)^2))\otimes_R R/I\longrightarrow\Omega^1_{R(G; \mathbb{Q})/\mathbb{Q}}\otimes_{R(G; \mathbb{Q})}\mathbb{Q}\longrightarrow\Omega^1_{R/\mathbb{Q}}\otimes_R R/I\longrightarrow 0.\]
	The image of the first map of the last sequence is isomorphic to $d(\text{ker}(i^*))\otimes_{R(G; \mathbb{Q})}\mathbb{Q}$, which in turn is isomorphic to $\delta(\text{ker}(i^*))$ by Remark \ref{deltaredux} and to $\delta^{G/K}(\text{ker}(i^*))$ by the isomorphism $j^*$ (Proposition \ref{kthrystructure}). By exactness we have 
	\[\text{dim}_\mathbb{Q}\delta^{G/K}(\text{ker}(i^*))+\text{dim}_\mathbb{Q}\Omega^1_{R/\mathbb{Q}}\otimes_R R/I=\text{dim}_\mathbb{Q}\Omega^1_{R(G; \mathbb{Q})/\mathbb{Q}}\otimes_{R(G; \mathbb{Q})}\mathbb{Q}.\]
	By Proposition \ref{repringreg}, the RHS of the above equation is $\text{rank }G$. 
	Now the proof of Theorem \ref{rep} is complete. 
\end{proof}
\begin{corollary}\label{regrep}
	The representation ring $R(G; \mathbb{Q})$ is regular at the augmentation ideal $I(G; \mathbb{Q})$. 
\end{corollary}
\begin{proof}
	We know that $(G, G)$ is always isotropy formal as the $G$-action on $G/G$ is trivial. By Theorem \ref{rep}, $R(G; \mathbb{Q})$ is regular at its augmentation ideal, and so is $R(G)$. 
\end{proof}
\begin{remark}
	In fact, the following string of equalities
	\begin{align*}
		&\text{dim}_\mathbb{Q}\Omega^1_{R(G; \mathbb{Q})/\mathbb{Q}}\otimes_{R(G; \mathbb{Q})}\mathbb{Q}\\
		=&\text{rank}G\ (\text{by Proposition \ref{repringreg}})\\
		=&\text{Krull dim} R(G; \mathbb{Q})\ (\text{by \cite[Lemma 7.12]{CFok}})\\
		=&\text{Krull dim} R(G; \mathbb{Q})_{I(G; \mathbb{Q})}\ (\text{by \cite[Corollary 11.27]{AM}})
	\end{align*} 	
	also imply that $R(G; \mathbb{Q})$ is regular at $I(G; \mathbb{Q})$ by \cite[Corollary 3.7]{ABC+}
\end{remark}
\begin{example}
	We will show by explicit computation that $R(\text{PSU}(3))$ is regular at the augmentation ideal, in accordance with our result and thus correct a claim made in \cite{Ho}. The representation ring $R(\text{PSU}(3))$ is isomorphic to 
	\[\mathbb{Z}[x, y, z]/(y^3-y^2-xz-2y(x+z)-x-y-z), \]
	where $x=[V_{3L_1}]$, $y=[V_{2L_1+L_2}]$ and $z=[V_{3L_1+3L_2}]$. As a check, one may make the linear change of variables $x= X+1-2Y$, $y=Y-1$ and $z=Z+1-2Y$ and obtain the more compact presentation in \cite[Lemma 7.1]{BZ}:
	\[R(\text{PSU}(3))\cong \mathbb{Z}[X, Y, Z]/(Y^3-XZ).\]
	Note that $\text{dim }V_{3L_1}=\text{dim }V_{3L_1+3L_2}=10$ and $\text{dim }V_{2L_1+L_2}=8$. If we let $\overline{x}=x-10$, $\overline{y}=y-8$ and $\overline{z}=z-10$, then 
	\[R(\text{PSU}(3))\cong\mathbb{Z}[\overline{x}, \overline{y}, \overline{z}]/(\overline{y}^3+23\overline{y}^2-2(\overline{x}+\overline{z})\overline{y}-\overline{x}\overline{z}-27(\overline{x}+\overline{z})+135\overline{y}).\]
	It follows that 
	\[I(\text{PSU}(3))/I(\text{PSU}(3))^2=\text{span}_\mathbb{Z}\{\overline{x}, \overline{y}, \overline{z}\}/\text{span}_\mathbb{Z}\{-27(\overline{x}+\overline{z})+135\overline{y}\}\]
	whose rank is 2, which is equal to the the Krull dimension of $R(\text{PSU}(3))$ (or equivalently the rank of $\text{PSU}(3)$ by \cite[Lemma 7.12(ii)]{CFok}). Thus $R(\text{PSU}(3))$ is indeed regular at its augmentation ideal. Moreover $I(\text{PSU}(3))/I(\text{PSU}(3))^2$ is an (non-free) abelian group of rank 2 instead of a free abelian group of rank 3 as claimed after \cite[Proof of Lemma 4.2]{Ho}. 
\end{example}


We are now ready to demonstrate the first application of Theorem \ref{rep} in this paper, which is to show Theorem \ref{circle}, a characterization of the isotropy formality of $(G, S)$ where $S$ is a circle subgroup, in terms of the self-duality of representations of $G$ when restricted to $S$.
\begin{proof}[Proof of Theorem \ref{circle}] Let $\rho_1, \rho_2, \cdots, \rho_m$ be representations which generate $R(G)$, and $\displaystyle i^*\rho_k=\sum_{n}a_{k, n}t^n\in R(S)$, where $t$ is one-dimensional standard representation of $S$. Suppose that any representation of $G$, when restricted to $S$, is self-dual. Then $a_{k, n}=a_{k, -n}$ for all $k$ and $n$. If we denote $t+t^{-1}-2$ by $s$, then $i^*\rho_k$ can be rewritten as a polynomial $p_k(s)$ of $s$. The regularity of $R:=i^*R(G; \mathbb{Q})$ at the augmentation ideal $i^*I(G; \mathbb{Q})=R\cap (t-1)$ amounts to the smoothness of $\text{Spec} R$ at the point $i^*I(G; \mathbb{Q})$. This is also equivalent to the smoothness of the image of the map $i_\mathbb{R}: \text{Spec} R(S; \mathbb{R})=\text{Spec}\mathbb{R}[t, t^{-1}]\to \text{Spec} R(G; \mathbb{R})=\text{Spec}\mathbb{R}[x_1, x_2, \cdots, x_{m}]/\mathcal{I}\subseteq\text{Spec}\mathbb{R}[x_1, x_2, \cdots, x_m]$ at the point $i_\mathbb{R}((t-1))=I(G; \mathbb{R})$ where $\mathcal{I}$ is the kernel of the map $\mathbb{R}[x_1, x_2, \cdots, x_m]\to R(G; \mathbb{R})$ which sends $x_i$ to $\rho_i$. In other words, the curve with the parametrization $s\mapsto (p_1(s), p_2(s), \cdots, p_m(s))\in\mathbb{R}^m$, which is a reparametrization of the curve $\displaystyle t\mapsto \left(\sum_n a_{1, n}t^n, \cdots, \sum_n a_{m, n}t^n\right)$, is smooth at the point $s=0$ (note that $t=1$ if and only if $s=0$). Thus it suffices to show that $(p_1'(0), p_2'(0), \cdots, p_m'(0))\neq 0$. Differentiating the equation $\displaystyle p_k(s)=\sum_n a_{k, n}t^n$ with respect to $t$ and applying chain rule to the LHS, we get $p_k'(0)$ as follows.
\begin{align*}
	\frac{d}{ds}p_k(s)\cdot\frac{ds}{dt}&=\sum_n na_{k, n}t^{n-1}\\
	\frac{d}{ds}p_k(s)\cdot (1-t^{-2})&=\sum_{n>0}na_{k, n}(t^{n-1}-t^{-n-1})\\
	\frac{d}{ds}p_k(s)&=\sum_{n>0}na_{k, n}(t^{n-1}+t^{n-3}+\cdots+t^{1-n})\\
	p_k'(0)&=\sum_{n>0}n^2a_{k, n}
\end{align*}
The RHS of the last line is not zero because $a_{k, n}\geq 0$ but there exists $n=n_0>0$ such that $a_{k, n_0}>0$. This shows that $R$ is regular at $i^*I(G; \mathbb{Q})$. 

Now suppose that there exists a representation of $G$ whose restriction to $S$ is not self-dual. Let $U$ be a set of representations of $G$ which generate $R(G; \mathbb{R})$, $\overline{U}$ the set of dual representations, and $U\cup \overline{U}=\{\rho_1, \rho_2, \cdots, \rho_m\}$. In this way $\{\rho_1, \rho_2, \cdots, \rho_m\}$ generates $R(G; \mathbb{Q})$ and has the nontrivial involution of taking dual (If $G$ is in addition simply-connected, the set $(\rho_1, \rho_2, \cdots, \rho_m)$ can be taken to be the set of fundamental representations as it generates $R(G; \mathbb{Q})$ and is invariant under the involution of taking dual). Let this involution send $\rho_k$ to $\rho_{\sigma(k)}$, where $\sigma$ is a nontrivial involution on $\{1, 2, \cdots, m\}$. We want to show that $R$ is not regular at $i^*I(G; \mathbb{Q})$. We have that $i^*\rho_k=q_k(t)$ for some Laurent polynomial $q_k$. By the same reasoning as in the previous paragraph, it suffices to show that the image curve $t\mapsto (q_1(t), q_2(t), \cdots, q_m(t))\subseteq \mathbb{R}^m$ for $t>0$ is not smooth at the point $t=1$. Suppose for the sake of contradiction that the curve is smooth at $t=1$. We shall first show that the map $t\mapsto (q_1(t), q_2(t), \cdots, q_m(t))$ is injective for $t$ in an open neighborhood of 1. Choose $k$ such that $k\neq \sigma(k)$. It suffices to show that the map $t\mapsto(q_k(t), q_{\sigma(k)}(t))$ is injective in an open neighborhood of 1. As $q_k(t)$ is a nonconstant Laurent polynomial, there exists $\delta_1>0$ such that $q'_k(t)>0$ or $q'_k(t)<0$ for $t\in(1, 1+\delta_1)$. So $q_k$ and hence $t\mapsto(q_k(t), q_{\sigma(k)}(t))$ is injective for $t\in(1, 1+\delta_1)$. Similarly, there exists $\delta_2>0$ such that $q_k(t)-q_{\sigma(k)}(t)>0$ (resp. $q_k(t)-q_{\sigma(k)}(t)<0$) for $t\in(1, 1+\delta_2)$. Take $\delta=\min\{\delta_1, \delta_2\}$. For $\displaystyle t\in\left(\frac{1}{1+\delta}, 1\right)$, $t^{-1}\in(1, 1+\delta)$, and the map $t\mapsto(q_k(t), q_{\sigma(k)}(t))=(q_{\sigma(k)}(t^{-1}), q_k(t^{-1}))$ is injective and $q_{\sigma(k)}(t^{-1})-q_k(t^{-1})<0$ (resp. $q_{\sigma(k)}(t^{-1})-q_k(t^{-1})>0$). Thus $t\mapsto (q_k(t), q_{\sigma(k)}(t))$ is injective for $\displaystyle t\in\left(\frac{1}{1+\delta}, 1+\delta\right)$. We may further restrict $\displaystyle\left(\frac{1}{1+\delta}, 1+\delta\right)$ to a smaller neighborhood of 1 so that the part of the curve for $t$ in that neighborhood is smooth. Now we can reparametrize that part of the curve by $(r_1(s), r_2(s), \cdots, r_m(s))$ so that the tangent vectors with respect to $s$ is nowhere vanishing (we may parametrize by arc length, for example). In particular the tangent vector at the point $(q_1(1), q_2(1), \cdots, q_m(1))=(\text{dim}\rho_1, \text{dim}\rho_2, \cdots, \text{dim}\rho_m)$ is nonzero. Let $s=s_0$ correspond to the point $(\text{dim}\rho_1, \text{dim}\rho_2, \cdots, \text{dim}\rho_m)$. Note that, since $G$ is semisimple, for any $k$ $q_k(t)$ is of the form $\displaystyle\sum_n a_{k, n}t^n$ where $\displaystyle\sum_n na_{k, n}=0$ and $a_{k, n}\geq 0$. By the AM-GM inequality, 
\[\displaystyle\sum_n a_{k, n}t^n\geq\left(\sum_n a_{k, n}\right)\sqrt{\prod_n (t^n)^{a_{k, n}}}=\text{dim}\rho_k,\] 
and equality holds if and only if $t=1$. Thus each coordinate $r_k(s)$ of the curve achieves its minimum at $s=s_0$ and so $r'(s_0)=(r'_1(s_0), r'_2(s_0), \cdots, r'_m(s_0))=(0, 0, \cdots, 0)$, contradicting the nonvanishing of the tangent vector at $(\text{dim}\rho_1, \text{dim}\rho_2, \cdots, \text{dim}\rho_m)$. It follows that the curve is not smooth at the point $(\text{dim}\rho_1, \text{dim}\rho_2, \cdots, \text{dim}\rho_m)$, which more precisely is a cusp given that each coordinate of the curve is bounded below by the corresponding coordinate of the point.
\end{proof}
	
\section{Characterization of isotropy formality of corank 1 isotropy actions}
\begin{theorem}\label{invtpoly}
	Let $G$ be a compact connected Lie group of rank $n$, $S$ a torus subgroup of dimension being $n-1$, $T$ a maximal torus of $G$ containing $S$, $\mu\in \chi(T)$ a primitive character of $T$ such that $S=\text{ker }\mu$, $L$ a primitive weight in the weight lattice $\mathfrak{t}^*_\mathbb{Z}$ corresponding to $\mu$ after exponentiation, and $W$ the Weyl group of $G$. Define
	\[\Phi_L:=\left(\prod_{w\in W/W_{[L]}} wL\right)^{n(L)}\in \text{Sym}^*(\mathfrak{t}^*)^W,\]
	where $W_{[L]}$ is the stabilizer of $[L]:=\{\pm L\}\in \mathfrak{t}^*/\{\pm 1\}$, $n(L)=1$ or 2 depending on whether $\displaystyle\prod_{w\in W/W_{[L]}}wL$ is $W$-invariant or $W$-semi-invariant, i.e. there exists a nontrivial one-dimensional real representation $\chi$ of $W$ such that
	\[\displaystyle w'(\prod_{w\in W/W_{[L]}}wL)=\chi(w')\prod_{w\in W/W_{[L]}}wL\text{ for all }w'\in W.\] 
	Then $(G, S)$ is isotropy formal if and only if $\Phi_L\notin (f_1, \cdots, f_n)^2$, where $f_1, \cdots, f_n$ are the fundamental $W$-invariant homogeneous polynomials on $\mathfrak{t}$, i.e. $\text{Sym}^*(\mathfrak{t}^*)^W=\mathbb{R}[f_1, \cdots, f_n]$.
\end{theorem}
\begin{proof}
	First, we use $\widehat{R}$ to denote the completion of $R$ at its augmentation ideal. Recall that the augmentation ideal of $\text{Sym}^*(V)$ is $\text{Sym}^{\geq 1}(V)$. Then with the identification $\text{Sym}^*(\mathfrak{t}^*)^W=\mathbb{R}[f_1, \cdots, f_n]$, the augmentation ideal of $\text{Sym}^*(\mathfrak{t}^*)^W$ is $(f_1, \cdots, f_n)$. Consider the following diagram.
	\begin{eqnarray*}
		\xymatrix{\widehat{R(G; \mathbb{R})}\ar[r]^{\text{ch}}\ar[d]^{i^*}& \widehat{\text{Sym}^*(\mathfrak{t}^*)^W}\ar[d]^{i^*}\\ \widehat{R(S; \mathbb{R})}\ar[r]^{\text{ch}}&\widehat{\text{Sym}^*(\mathfrak{s}^*)}}
	\end{eqnarray*}
	The horizontal maps are Chern character maps from the completed equivariant $K$-theory coefficient ring (identified with the completed representation ring) to the completed equivariant cohomology coefficient ring (identified with the completed polynomial ring on the Lie algebra of the acting group). By \cite[Theorem 5.3]{CFok}, the Chern character map is a ring isomorphism. The diagram commutes due to the functoriality of the Chern character. By Theorem \ref{rep}, $(G, S)$ is isotropy formal if and only if $i^*R(G; \mathbb{R})$ is regular at $i^*I(G; \mathbb{R})$. The latter condition then is equivalent to $i^*\widehat{R(G; \mathbb{R})}$ being regular at $i^*I(G; \mathbb{R})$, which in turn is equivalent to $\text{ker}(i^*: \widehat{R(G; \mathbb{R})}\to\widehat{R(S; \mathbb{R})})\nsubseteq I(G; \mathbb{R})^2$. Applying the Chern character and commutativity of the above diagram, we have that the last condition is equivalent to $\mathfrak{p}:=\text{ker}(i^*: \widehat{\text{Sym}^*(\mathfrak{t}^*)^W}\to \widehat{\text{Sym}^*(\mathfrak{s}^*)})\nsubseteq (f_1, \cdots, f_n)^2$. Note that
	\begin{align*}
		\text{Krull dim }\widehat{\text{Sym}^*(\mathfrak{t}^*)^W}&=\text{Krull dim }\widehat{\text{Sym}^*(\mathfrak{t}^*})\\
													&(\text{because }\widehat{\text{Sym}^*(\mathfrak{t}^*)}\text{ is integral over }\widehat{\text{Sym}^*(\mathfrak{t}^*)^W}\text{ by \cite[Exercise 5.12]{AM}})\\
													&=\text{dim }T\\
													&=\text{rank }G\\
													&=n\\
		\text{Krull dim }\widehat{\text{Sym}^*(\mathfrak{s}^*)}&=\text{dim }S\\
												&=n-1\\
		\text{Krull dim }i^*\widehat{\text{Sym}^*(\mathfrak{t}^*)^W}&=\text{Krull dim }i^*\widehat{R(G; \mathbb{R})}\\
													&=\text{Krull dim }\widehat{R(S; \mathbb{R})}\\
													&(\text{because }\widehat{R(S; \mathbb{R})}\text{ is integral over }i^*\widehat{R(G; \mathbb{R})}\text{ by \cite[Proposition 3.2]{Se}})\\
													&=\text{dim }S\\
													&=n-1.
	\end{align*}
	It follows that $\mathfrak{p}$ is a prime ideal of height 1 of $\widehat{\text{Sym}^*(\mathfrak{t}^*)^W}$. Since $\widehat{\text{Sym}^*(\mathfrak{t}^*)^W}$ is a Noetherian UFD, $\mathfrak{p}$ is a principal ideal (\cite[Lemma 10.120.6]{SP}). Let $\Phi_L$ be a generator of $\mathfrak{p}$. Now isotropy formality of $(G, S)$ is equivalent to $\Phi_L\nsubseteq (f_1, \cdots, f_n)^2$. To finish the proof, it remains to show that $\Phi_L$ can be chosen to be the polynomial in the proposition. Note that 
	\[\mathfrak{p}=(\Phi_L)\subseteq \text{ker}(i^*: \widehat{\text{Sym}^*(\mathfrak{t}^*)}\to \widehat{\text{Sym}^*(\mathfrak{s}^*)})=(L).\]
	The vanishing set $V(L)$ is the hyperplane $\text{ker }L$ in $\mathfrak{t}$. Since $\Phi_L$ is $W$-invariant, $V(\Phi_L)$ should contain $V(L)$ and be $W$-invariant in $\mathfrak{t}$. Thus $V(\Phi_L)$ is the $W$-orbit of the hyperplane $\text{ker }L$, i.e. 
	\[V(\Phi_L)=\bigcup_{w\in W/W_{[L]}}\text{ker }wL.\]
	Thus $\Phi_L$ should be a $W$-invariant polynomial of minimal degree on $\mathfrak{t}$ which vanishes on $\displaystyle \bigcup_{w\in W/W_{[L]}}\text{ker }wL$. The polynomial $\displaystyle \left(\prod_{w\in W/W_{[L]}} wL\right)^{n(L)}$ satisfies these conditions and this completes the proof.
\end{proof}
The following lemma, which is a consequence of Theorem \ref{invtpoly}, stipulates a bound on the size of the Weyl orbit of $[L]$ by the maximum degree of the fundamental $W$-invariant homogeneous polynomials in order for isotropy formality to hold. This places a severe restriction on $L$ which will prove very helpful in the classification of isotropy formality of corank 1 isotropy actions. 
\begin{lemma}\label{degreebound}
	Assume the same conditions as in Theorem \ref{invtpoly}. If $(G, S)$ is isotropy formal, then 
	\[n(L)|W/W_{[L]}|\leq \max\{\text{deg }f_1, \cdots, \text{deg }f_n\}.\]
\end{lemma}
\begin{proof}
	The LHS of the inequality is the (homogeneous) degree of $\Phi_L$. The condition $\Phi_L\nsubseteq (f_1, \cdots, f_n)^2$ is equivalent to $\Phi_L$ having nonzero linear terms of $f_1, \cdots, f_n$ after being rewritten as a polynomial of $f_1, \cdots, f_n$. Thus the degree of $\Phi_L$ must be equal to the degree of one of the fundamental $W$-invariant polynomials. The lemma then follows.
\end{proof}
\begin{proof}[Proof of Theorem \ref{corank1}]
	\begin{enumerate}
		\item $A_n$: In this case, $\displaystyle \mathfrak{t}^*_\mathbb{Z}=\left.\left\{\sum_{i=1}^{n+1}a_i L_i\right| \sum_{i=1}^{n+1}a_i=0\right\}$, $\displaystyle \omega_i=\sum_{j=1}^i L_j$, $W=S_{n+1}$ which permutes $L_i$, $\text{deg }f_i=i+1$ ($f_i$ is the elementary symmetric polynomial $e_{i+1}$ of degree $i+1$) and so the maximum fundamental degree is $n+1$. First, we shall verify that $L=\omega_1=L_1$ gives rise to an isotropy formal action. Note that $\Phi_L=L_1L_2\cdots L_{n+1}$, which is $f_n$. So $\Phi_L\nsubseteq (f_1, \cdots, f_n)^2$ and $(G, S)$ is isotropy formal by Theorem \ref{invtpoly}. When $n=3$, $L=\omega_2=L_1+L_2$ also leads to an isotropy formal action. That is because $\Phi_L=(L_1+L_2)(L_1+L_3)(L_1+L_4)$ (the stabilizer $W_{[L]}$ also include the permutation $(13)(24)$ which sends $L$ to $-L$ and hence stabilizes $[L]$), which is $f_2$ and thus not in $(f_1, f_2, f_3)^2$.
		
		Now we shall prove that there are no other primitive weights which give rise to isotropy formal actions. Let $\displaystyle L=\sum_{i=1}^{n+1}a_i L_i$ and $m_1, m_2, \cdots, m_r$ the multiplicities of the coefficients $a_1, \cdots, a_{n+1}$. If there are no Weyl group elements which negate $L$, then $\displaystyle |W/W_{[L]}|=\frac{(n+1)!}{m_1!\cdots m_r!}$. If the multiplicity type $\{m_1, \cdots, m_r\}\neq \{1, n\}$, then the multiplicity type with the largest stabilizer is $\{2, n-1\}$. Note that
		\[|W/W_{[L]}|\geq \frac{(n+1)!}{2!(n-1)!}=\binom{n+1}{2}=\frac{(n+1)n}{2}\]
	and the last term is greater than $\max\{\text{deg }f_1, \cdots, \text{deg }f_n\}=n+1$ when $n\geq 3$. So according to Lemma \ref{degreebound}, when $n\geq 3$ and $L$ has the multiplicity type not equal to $\{1, n\}$ and there are no Weyl group elements taking $L$ to $-L$, then $L$ does not give rise to an isotropy formal pair.
		
	If there are Weyl group elements which take $L$ to $-L$, then either the nonzero coefficients of $L$ form pairs of numbers negative of each other, or $L$ is of the form $\sum_{i=1}^{\frac{n+1}{2}}L_i$ when $n+1$ is even. For the former case, the multiplicity type with the largest stabilizer is $\{1, 1, n-1\}$ ($L=L_1-L_2$ is an example with such a multiplicity type). Note that
	\[|W/W_{[L]}|\geq\frac{(n+1)!}{2(n-1)!}=\binom{n+1}{2}=\frac{(n+1)n}{2}\] 
	and again the last term is greater than $\max\{\text{deg }f_1, \cdots, \text{deg }f_n\}=n+1$ when $n\geq 3$. For the latter case, the multiplicity type is $\{\frac{n+1}{2}, \frac{n+1}{2}\}$. Then 
	\[|W/W_{[L]}|=\frac{1}{2}\binom{n+1}{\frac{n+1}{2}}\]
	which is greater than $\max\{\text{deg }f_1, \cdots, \text{deg }f_n\}=n+1$ when $n\geq 5$ and $n$ is odd. 
	
	In sum, using Lemma \ref{degreebound}, we rule out isotropy formality in the following cases.
	\begin{enumerate}
		\item $n\geq 3$, the multiplicity type of $L$ is not $\{1, n\}$ and no Weyl group elements take $L$ to $-L$.
		\item $n\geq 3$, the coefficients of $L$ forms pairs of numbers negative of each other.
		\item $n\geq 5$ and $n$ is odd, and $\displaystyle L=\sum_{i=1}^{\frac{n+1}{2}}L_i$.
	\end{enumerate}
	We are then left with cases $n=1$, 2 and 3. For $n=1$, then $S$ is the trivial subgroup and $(G, S)$ is trivially isotropy formal. For $n=2$ or 3, we can apply the above arguments to show that $L$ has to be one of the weights listed in the table in Theorem \ref{corank1} in order for $(G, S)$ to be isotropy formal. When $n=2$ and $L$ is not equal to any one of the weights listed in the table in Theorem \ref{corank1}, i.e. $L\in W\cdot(a_1L_1+a_2L_2)$ for $a_1, a_2$ distinct coprime integers, the degree of $\Phi_L$ is 6, which is greater than 3, the maximal fundamental degree. When $n=3$ and $L$ is not any one of the weights in the table in Theorem \ref{corank1}, i.e. $L\in W\cdot(a_1L_1+a_2L_2)$ for $a_1, a_2$ distinct coprime integers, or $L\in W\cdot (a_1L_1+a_2L_2+a_3L_3)$ where $a_i$'s are nonzero and not all the same and $L$ is primitive, the degree of $\Phi_L$ is at least 12 (for example, if $L=L_1-L_2$ or $L=L_1+L_2+2L_3=L_3-L_4$, the degree of $\Phi_L$ is at least $\displaystyle\frac{4!}{2}=12$ as $[L]$ is stabilized by a transposition on $L_i$;s), which is greater than 4, the maximum fundamental degree. In these cases $(G, S)$ is not isotropy formal by Lemma \ref{degreebound}.
	\item $B_n$ and $C_n$: In both cases, $\mathfrak{t}^*_\mathbb{Z}=\text{span}_\mathbb{Z}\{L_1, L_2, \cdots, L_{n-1}, k(L_1+L_2+\cdots+L_n)\}$, where $k=\frac{1}{2}$ for $B_n$ and $k=1$ for $C_n$. The fundamental weights of $B_n$ are $\displaystyle \omega_i=\sum_{j=1}^i L_j$ for $1\leq i\leq n-1$ and $\displaystyle \omega_n=\frac{1}{2}\sum_{i=1}^n L_i$, whereas those of $C_n$ are the same except that $\displaystyle \omega_n=\sum_{i=1}^n L_i$. The Weyl group $W$ for both cases is the signed permutation group $\mathbb{Z}_2^n\rtimes S_n$, where $S_n$ permutes $L_1, \cdots, L_n$ and the $i$-th generator of $\mathbb{Z}_2^n$ negates $L_i$ for $1\leq i\leq n$. The fundamental invariant polynomials of $\text{Sym}^*(\mathfrak{t}^*)^W$ are $f_i:=e_i(L_1^2, L_2^2, \cdots, L_n^2)$ for $1\leq i\leq n$, where $e_i$ is the $i$-th elementary symmetric polynomial. Then the maximal fundamental degree is $2n$. 
	
	When $n=2$, $S$ is a circle subgroup. For any primitive weight $L=a_1L_1+a_2L_2$, we have 
	\[\Phi_L=(a_1L_1+a_2L_2)(a_1L_1-a_2L_2)(a_2L_1+a_1L_2)(a_2L_1-a_1L_2)=a_1^2a_2^2f_1^2-(a_1^2+a_2^2)^2f_2\]
	which is not in $(f_1, f_2)^2$. By Theorem \ref{invtpoly}, $(G, S)$ is isotropy formal for any circle subgroup $S$. Alternatively, note that any representations of $B_2$ and $C_2$ are self-dual and so are their restricted representations to $S$. By Theorem \ref{circle}, $(G, S)$ is isotropy formal for any circle subgroup $S$.
	
	Let $n\geq 3$, $\displaystyle L=\sum_{i=1}^n a_i L_i$, $k$ the number of nonzero coefficients among $a_1, \cdots, a_n$, and $m_1, \cdots, m_j$ the multiplicities of the absolute values of the nonzero coefficients. If $k=1$, i.e. $L\in W\cdot\omega_1$, then 
	\[\Phi_L=L_1^2L_2^2\cdots L_n^2=f_n\notin (f_1, f_2, \cdots, f_n)^2.\]
	It follows that when $L\in W\cdot\omega_1$, $(G, S)$ is isotropy formal by Theorem \ref{invtpoly}. In general, if $\displaystyle L':=\sum_{i=1}^n |a_i|L_i$, then $[L']$ is in the $W$-orbit of $[L]$, and hence
	\[|W/W_{[L]}|=|W/W_{[L']}|=\frac{2^n n!}{2\cdot2^{n-k}m_1!m_2!\cdots m_j!(n-k)!}\geq 2^{k-1}\binom{n}{k}\]
	(the stabilizer $W_{[L']}$ is generated by the negation of $L'$, the permutations of the basis vectors sharing the same coefficient, as well as the sign flips of those basis vectors whose coefficients are 0). When $n\geq 5$ and $k\geq 2$, the degree of $\Phi_L$ is at least $\displaystyle 2^{k-1}\binom{n}{k}$, which is greater than or equal to 
	\[\begin{cases}2\binom{n}{2}=n(n-1)&,\ \text{if }2\leq k\leq n-2, \\ 2^{n-2}n&,\ \text{if }k=n-1,\\ 2^{n-1}&,\ \text{if }k=n.\end{cases}\]
	These numbers are all strictly greater than the maximal fundamental degree $2n$. By Lemma \ref{degreebound}, $(G, S)$ is not isotropy formal for $n\geq 5$ and $k\geq 2$.
	
	It remains to discuss the cases when $n=3, 4$ and $k\geq 2$. If $n=3$ and $k=2$, i.e. $L\in W\cdot \omega_2$, then 
	\[\Phi_L=((L_1+L_2)(L_2+L_3)(L_1+L_3)(L_1-L_2)(L_2-L_3)(L_1-L_3))^2.\]
	Thus $\text{deg }\Phi_L=12$, which is greater than the maximal fundamental degree $2\cdot 3=6$. By Lemma \ref{degreebound}, $(G, S)$ is not isotropy formal. If $n=3$ and $k=3$, i.e. $L\in W\cdot\omega_3$, 
	\[\Phi_L=\begin{cases}\frac{1}{16}\prod_{\varepsilon_1, \varepsilon_2\in\{\pm 1\}}(\varepsilon_1L_1+\varepsilon_2L_2+L_3)=\frac{1}{16}(f_1^2-4f_2)&\ \text{for }B_3,\\ \prod_{\varepsilon_1, \varepsilon_2\in\{\pm 1\}}(\varepsilon_1L_1+\varepsilon_2L_2+L_3)=f_1^2-4f_2&\ \text{for }C_3, \end{cases}\]
	which is not in $(f_1, f_2, f_3)^2$. By Theorem \ref{invtpoly}, $(G, S)$ is isotropy formal if $n=3$ and $L\in W\cdot\omega_3$. When $n=4$ and $k=2$ or 3, the degree of $\Phi_L$ is at least $2^{k-1}\binom{n}{k}$, which is strictly greater than the maximal fundamental degree $2n=8$. By Lemma \ref{degreebound}, $(G, S)$ is not isotropy formal. When $n=k=4$, i.e. $L\in W\cdot\omega_4$, 
	\[\Phi_L=\begin{cases}\frac{1}{256}\prod_{\varepsilon_1, \varepsilon_2, \varepsilon_3\in\{\pm 1\}}(\varepsilon_1L_1+\varepsilon_2L_2+\varepsilon_3L_3+L_4)=\frac{f_1^4-8f_1^2f_2+16f_2^2-64f_4}{256}&\ \text{for }B_4,\\ \prod_{\varepsilon_1, \varepsilon_2, \varepsilon_3\in\{\pm 1\}}(\varepsilon_1L_1+\varepsilon_2L_2+\varepsilon_3L_3+L_4)=f_1^4-8f_1^2f_2+16f_2^2-64f_4&\ \text{for }C_4,\end{cases}\]
	which is not in $(f_1, f_2, f_3, f_4)^2$. By Theorem \ref{invtpoly}, $(G, S)$ is isotropy formal if $n=4$ and $L\in W\cdot \omega_4$.
	\item $D_n$, $n\geq 4$: In this case, $\mathfrak{t}_\mathbb{Z}^*=\text{span}_\mathbb{Z}\left\{L_1, L_2, \cdots, L_{n-2}, \frac{L_1+L_2+\cdots+L_{n-1}\pm L_n}{2}\right\}$. The fundamental weights are $\displaystyle\omega_i=\sum_{j=1}^i L_j$ for $1\leq i\leq n-2$, $\displaystyle\omega_{n-1}=\frac{L_1+L_2+\cdots+L_{n-1}-L_n}{2}$ and $\displaystyle \omega_n=\frac{L_1+L_2+\cdots+L_{n-1}+L_n}{2}$. The Weyl group $W$ is generated by permutations and even number of sign flips on $\{L_1, L_2, \cdots, L_n\}$ and so $|W|=2^{n-1}n!$. The fundamental invariant polynomials of $\text{Sym}^*(\mathfrak{t}^*)^W$ are $f_i:=e_i(L_1^2, \cdots, L_n^2)$ for $1\leq i\leq n-1$ and $f_n:=L_1L_2\cdots L_n$. It follows that the maximal fundamental degree is $2n-2$.  Let $\displaystyle L=\sum_{i=1}^n a_i L_i$, $k$ the number of nonzero coefficients among $a_1, \cdots, a_n$, and $m_1, \cdots, m_j$ the multiplicities of the absolute values of the nonzero coefficients. If $k=1$, i.e. $L\in W\cdot\omega_1$, then 
	\[\Phi_L=L_1L_2\cdots L_n=f_n\notin (f_1, f_2, \cdots, f_n)^2.\]
	So when $L\in W\cdot\omega_1$, $(G, S)$ is isotropy formal by Theorem \ref{invtpoly}. In general, the size of the $W$-orbit of $[L]$ is  
	\[|W/W_{[L]}|=\begin{cases}\frac{2^{n-1} n!}{2\cdot 2^{n-k-1}m_1!m_2!\cdots m_j!(n-k)!}&,\ \text{if }1\leq k<n, \\ \frac{2^{n-1}n!}{2m_1!m_2!\cdots m_j!}&,\ \text{if }k=n\text{ and }n\text{ is even},\\ \frac{2^{n-1}n!}{m_1!m_2!\cdots m_j!}&,\ \text{if }k=n\text{ and }n\text{ is odd.}\end{cases}\]
	When $n\geq 5$ and $k\geq 2$, the above three expressions are greater than or equal to $2^{k-1}\binom{n}{k}$, $2^{n-2}$ and $2^{n-1}$ respectively, which in turn are greater than the maximal fundamental degree $2n-2$. By Lemma \ref{degreebound}, $(G, S)$ is not isotropy formal for $n\geq 5$ and $k\geq 2$. It remains to check the cases where $n=4$ and $k\geq 2$. When $k=2$ (respectively $k=3$), $\displaystyle |W/W_{[L]}|\geq 2^{k-1}\binom{n}{k}=12$ (respectively $\displaystyle |W/W_{[L]}|\geq 2^{k-1}\binom{n}{k}=16$), which is greater than the maximal fundamental degree $2n-2=6$. Again by Lemma \ref{degreebound}, $(G, S)$ is not isotropy formal when $n=4$ and $k=2, 3$. When $k=4$, $\displaystyle |W/W_{[L]}|=\frac{2^{n-2}n!}{m_1!\cdots m_j!}=4\binom{4}{m_1\ \cdots\ m_j}$, which does not exceed $2n-2=6$ if and only if $m_{i_0}=4$ for some $i_0$ and $m_j=0$ for $j\neq i_0$, i.e. $|a_i|$ are all equal for $1\leq i\leq 4$. Up to the Weyl group action, there are two possibilities for such a primitive weight $L$, namely $L=\omega_3$ or $L=\omega_4$. We have
	\[\Phi_L=\begin{cases}\frac{1}{16}\prod_{\varepsilon_i\in\{\pm 1\}, \varepsilon_1\varepsilon_2\varepsilon_3=1}(\varepsilon_1L_1+\varepsilon_2L_2+\varepsilon_3L_3-L_4)=\frac{1}{16}(f_1^2-4f_2-8f_4)&,\ \text{if }L=\omega_3,\\ \frac{1}{16}\prod_{\varepsilon_i\in\{\pm 1\}, \varepsilon_1\varepsilon_2\varepsilon_3=1}(\varepsilon_1L_1+\varepsilon_2L_2+\varepsilon_3L_3+L_4)=\frac{1}{16}(f_1^2-4f_2+8f_4)&,\ \text{if }L=\omega_4.\end{cases}\]
	In both cases, $\Phi_L\notin (f_1, \cdots, f_4)^2$. By Theorem \ref{invtpoly}, $(G, S)$ is isotropy formal for $G$ being of type $D_4$ and $L=\omega_3$ or $\omega_4$.
	\item $G_2$: In this case the dimension of $S$ is 1. Note that any representations of $G_2$ are self-dual (cf. \cite{Bo}) and so are their restricted representations to $S$. By Theorem \ref{circle}, $(G, S)$ is isotropy formal for any circle subgroup $S$.
	
	Alternatively, we can also use Theorem \ref{invtpoly} to show isotropy formality. Identify the Lie algebra of $G_2$ with $V:=\{(x_1, x_2, x_3)\in\mathbb{R}^3|\ x_1+x_2+x_3=0\}$. Then the fundamental weights $\omega_1$ and $\omega_2$ of $G_2$ can be identified with $x_1-x_2$ and $2x_1-x_2-x_3$ respectively, and thus the weight lattice of $G_2$ with $\text{span}_\mathbb{Z}\{x_1-x_2, 2x_1-x_2-x_3\}$. The Weyl group of $G_2$, which is isomorphic to $S_3\times\mathbb{Z}_2$, acts on $V$ by permuting the three coordinates and simultaneous negation of the coordinates. Then the fundamental invariant polynomials are $f_1=x_1^2+x_2^2+x_3^2$ and $f_2=x_1^2x_2^2x_3^2$. Let $L=a\omega_1+b\omega_2=(a+2b)x_1-bx_2-(a+b)x_3$ be a primitive weight. Then 
	\begin{align*}
		\Phi_L&=\prod_{\sigma\in S_3}((a+2b)x_{\sigma(1)}-bx_{\sigma(2)}-(a+b)x_{\sigma(3)})\\
			&=27(a^2+3ab+3b^2)f_2-\frac{a^2(a+3b)^2(2a+3b)^2}{8}f_1^3
	\end{align*}
	which is not in $(f_1, f_2)^2$. By Theorem \ref{invtpoly}, we again obtain the conclusion that $(G, S)$ is isotropy formal for any circle subgroup $S$.
	\end{enumerate}
	As an interlude, let us prove a general result about the size of the smallest $W$-orbit of a weight for a general compact, connected simple Lie group $G$, which will serve as the common argument for the non-existence of isotropy formal corank 1 actions for the remaining exceptional cases. Let $\displaystyle L=\sum_{i=1}^n a_i\omega_i$ and $s_i$ be the reflection across the hyperplane orthogonal to the simple root $\alpha_i$. Note that $W$ is generated by $s_i$ for $1\leq i\leq n$, and that $s_i$ stabilizes $L$ as a weight if and only if $L-\langle L, \alpha_i^\vee\rangle\alpha_i=L$, i.e. $\langle L, \alpha_i^\vee\rangle=0$, or equivalently $a_i=0$. Thus the stabilizer of $L$ as a weight is the subgroup $W_{\{1, \cdots, n\}\setminus\{i|\ a_i=0\}}:=\langle s_i|\ a_i=0\rangle$. Then if the size of the $W$-orbit of $L$ as a weight is the smallest, then only one coefficient among $a_1, \cdots, a_n$ is nonzero. In other words, the $W$-orbits of the fundamental weights are candidates for the smallest orbit. The stabilizer $W_{\{i\}}$ of $\omega_i$ is in fact the Weyl group of the maximal parabolic subgroup of the complexified group $G^\mathbb{C}$ corresponding to the Dynkin diagram obtained by deleting the $i$-th node from that of $G^\mathbb{C}$. 
	\begin{enumerate}
		\setcounter{enumi}{4}
		\item $F_4$: The sizes of the $W$-orbits of the fundamental weights of $F_4$ are as follows.
		\begin{align*}
			|W\cdot\omega_1|&=\frac{|W(F_4)|}{|W(B_3)|}=\frac{1152}{2^3\cdot 3!}=24\\
			|W\cdot\omega_2|&=|W\cdot\omega_3|=\frac{|W(F_4)|}{|W(A_1\times A_2)|}=\frac{1152}{2\cdot 6}=96\\
			|W\cdot\omega_4|&=\frac{|W(F_4)|}{|W(C_3)|}=\frac{1152}{2^3\cdot 3!}=24
		\end{align*}
		Thus the $W$-orbits of $\omega_1$ and $\omega_4$ are the smallest. Since $-\text{Id}_{\mathfrak{t}^*}\in W(F_4)$ (cf. \cite{Bo}), $|W(F_4)\cdot[\omega_i]|=\frac{24}{2}=12$ for $i=1, 4$. Recall that $\omega_1=L_1+L_2$ and $\omega_4=L1$, and the four simple roots of $F_4$ are 
		\[\alpha_1=L_2-L_3,\ \alpha=L_3-L_4,\ \alpha_3=L_4,\ \alpha_4=\frac{L_1-L_2-L_3-L_4}{2},\]
		where $\{L_i|\ 1\leq i\leq 4\}$ is an orthonormal basis of the weight lattice $\mathfrak{t}^*_{\mathbb{Z}}$. Since $\omega_1$ is a long root while $\omega_4$ is a short root, $W\cdot\omega_1$ is the set of long roots $\{\pm L_i\pm L_j|\ 1\leq i<j\leq 4\}$ while $W\cdot\omega_4$ is the set of short roots $\displaystyle \{\pm L_i|\ 1\leq i\leq 4\}\cup\left\{\frac{\pm L_1\pm L_2\pm L_3\pm L_4}{2}\right\}$. Choosing $\{L_i\pm L_j|\ 1\leq i<j\leq 4\}$ and $\displaystyle \{L_i|\ 1\leq i\leq 4\}\cup \left\{\frac{L_1\pm L_2\pm L_3\pm L_4}{2}\right\}$ as the sets of representatives of $W\cdot[\omega_1]$ and $W\cdot[\omega_4]$ respectively, we form the products
		\begin{align*}
			\Phi_1&=\prod_{1\leq i<j\leq 4}(L_i+L_j)\prod_{1\leq i<j\leq 4}(L_i-L_j),\\
			\Phi_4&=\frac{L_1L_2L_3L_4}{256}\prod_{\varepsilon_2, \varepsilon_3, \varepsilon_4\in\{\pm 1\}}(L_1+\varepsilon_2 L_2+\varepsilon_3 L_3+\varepsilon_4 L_4),
		\end{align*}
		whose degrees are 12. These polynomials are semi-invariant under the Weyl group action. For example, $s_{\alpha_1}$ swaps $L_2$ and $L_3$ and fixes $L_1$ and $L_4$. Then $s_{\alpha_1}\Phi_1=-\Phi_1$. Similarly, $s_{\alpha_3}$ negates $L_4$ and fixes $L_1$, $L_2$ and $L_3$. Thus $s_{\alpha_3}\Phi_4=-\Phi_4$. It follows that the degrees of $\Phi_{\omega_i}$ for $i=1$ and $4$ are both $12\cdot 2=24$. However, $\max\{\text{deg }f_1, \cdots, \text{deg }f_4\}=12$. By Lemma \ref{degreebound}, there does not exist any isotropy formal action of corank 1 when $G$ is of type $F_4$.
		\item $E_6$: The smallest $W$-orbit of a weight of $E_6$ is the orbit of $\omega_1$. Its size is
		\[\frac{|W(E_6)|}{|W(D_5)|}=\frac{51840}{1920}=27.\]
		Since $-\text{Id}_{\mathfrak{t}^*}\notin W(E_6)$, the orbit of $\omega_1$ as a weight is also the $W$-orbit of $[\omega_1]$. The degree of $\Phi_{\omega_1}$ is at least 27 which is greater than $\max\{\text{deg }f_1, \cdots, \text{deg }f_6\}=12$. By Lemma \ref{degreebound}, there does not exist any isotropy formal action of corank 1 when $G=E_6$.
		\item $E_7$: The smallest $W$-orbit of a weight of $E_7$ is the orbit of $\omega_7$. Its size is
		\[\frac{|W(E_7)|}{|W(E_6)|}=\frac{2903040}{51840}=56.\]
		Since $-\text{Id}_{\mathfrak{t}^*}\in W(E_7)$ (cf. \cite{Bo}), $|W(E_7)\cdot[\omega_7]|=\frac{56}{2}=28$. The degree of $\Phi_{\omega_7}$ is at least 28 which is greater than $\max\{\text{deg }f_1, \cdots, \text{deg }f_7\}=18$. By Lemma \ref{degreebound}, there does not exist any isotropy formal action of corank 1 when $G=E_7$.
		\item $E_8$: The smallest $W$-orbit of a weight of $E_8$ is the orbit of $\omega_8$. Its size is
		\[\frac{|W(E_8)|}{|W(E_7)|}=\frac{696729600}{2903040}=240.\]
		Since $-\text{Id}_{\mathfrak{t}^*}\in W(E_8)$ (cf. \cite{Bo}), $|W(E_8)\cdot[\omega_8]|=\frac{240}{2}=120$. The degree of $\Phi_{\omega_8}$ is at least 120 which is greater than $\max\{\text{deg }f_1, \cdots, \text{deg }f_8\}=30$. By Lemma \ref{degreebound}, there does not exist any isotropy formal action of corank 1 when $G=E_8$.
	\end{enumerate}
\end{proof}
\section{Characterization of isotropy formality using index pairing}
Apart from Theorem \ref{rep}, we also characterize isotropy formality in terms of a certain nondegeneracy of the equivariant $K$-theoretic index pairing of two special equivariant $K$-theory classes. This section is devoted to the proof of Theorem \ref{indexthry}, as well as a demonstration of its applicability to two examples.
	\begin{proposition}\label{nondegenformal}
		Let $G$ be a connected compact Lie group, $S$ a torus subgroup of $G$, and $\pi: G/S\to \text{pt}$ the collapsing map. Then $G/S$ is formal if and only if there exist $a\in \text{im }\alpha$ (see Proposition \ref{kthrystructure}) and $b\in \bigwedge\nolimits^{\text{dim }P}P$ such that the $K$-theoretic index pairing $\pi_*(ab)$ is nonzero.
	\end{proposition}
	\begin{proof}
		First, the tangent bundle $TG/S$ is isomorphic to $G\times_S\mathfrak{g}/\mathfrak{s}=G\times_S\left(\mathfrak{t}/\mathfrak{s}\oplus\bigoplus_{\alpha\in R^+}\mathfrak{g}_\alpha^{\mathbb{R}}\right)$, where $\mathfrak{g}_\alpha^{\mathbb{R}}=\left(\mathfrak{g}_\alpha^\mathbb{C}\oplus\mathfrak{g}_{-\alpha}^\mathbb{C}\right)\cap \mathfrak{g}$ (where $\mathfrak{g}_\alpha^\mathbb{C}$ is the root space corresponding to $\alpha$ in $\mathfrak{g}\otimes_\mathbb{R}\mathbb{C}$) for $\alpha\in R^+$ determine an $S$-invariant stable complex structure of the tangent bundle. Thus $G/S$ is $S$-equivariantly stably almost complex and hence $S$-equivariantly $\text{Spin}^c$, and it makes sense to consider the $K$-theoretic pushforward $\pi_*$. Suppose $G/K$ is formal. Choose any nonzero element $b$ of $\bigwedge\nolimits^{\text{dim }P}P$. By the nondegeneracy of the index pairing of any compact $\text{Spin}^c$ manifold, there exists $\overline{a}\in K^*(G/S)\otimes\mathbb{Q}$ such that $\pi_*(\overline{a}b)\neq 0$. By Proposition \ref{kthrystructure}, $\overline{a}=a_0+\sum_{i=1}^n a_ib_i$, where $a_i\in \text{im }\alpha$ for $0\leq i\leq n$ and $0\neq b_i\in \bigwedge\nolimits^kP$ for $1\leq i\leq n$ and $0<k\leq\text{dim }P$. It follows that $\pi_*(a_0b)=\pi_*(\overline{a}b)\neq 0$.
		
		To prove the converse, it suffices to focus on the case where $\text{dim }S<\text{rank }G$ because $G/S$ is formal if $S$ is a maximal torus of $G$. Now we assume that there exists $a\in A$ and $b\in \bigwedge\nolimits^{\text{dim }P}P$ such that $\pi_*(ab)\neq 0$. By the Atiyah-Singer index theorem, 
		\[\pi_*(ab)=\int_{G/S}\text{ch}(ab)\text{td}(G/S).\] 
		So $\text{ch}(ab)\text{td}(G/S)$ contains a term which is a ``volume form'' in $H^{\text{dim }G/S}(G/S; \mathbb{Q})$. On the other hand, consider the Cartan algebra $(H^*(BS; \mathbb{Q})\otimes H^*(G; \mathbb{Q}), d)=(\text{Sym}^*(\mathfrak{s}^*)\otimes \bigwedge\nolimits^*V_H, d)$, which is a minimal model for $H^*(G/S; \mathbb{Q})$. Here the differential $d$ satisfies $d(\bigwedge\nolimits^*V_H)\subseteq \text{Sym}^*(\mathfrak{s}^*)$ and $d(\text{Sym}^*(\mathfrak{s}^*))=0$. The exterior degree on $\bigwedge\nolimits^*V_H$ induces an exterior degree on the complex $(\text{Sym}^*(\mathfrak{s}^*)\otimes\bigwedge\nolimits^*V_H)_p:=\text{Sym}^*(\mathfrak{s}^*)\otimes\bigwedge\nolimits^p V_H$ and hence on its cohomology. Let 
		\[P_H:=\{p\in V_H| dp\in \text{Sym}^{\geq 1}(\mathfrak{s}^*)\cdot dV_H\}\]
		and $P'_H$ a graded linear complement of $P_H$ in $V_H$. Then 
		\begin{align*}
			H^*(G/S; \mathbb{Q})&\cong H^*(\text{Sym}^*(\mathfrak{s}^*)\otimes\bigwedge\nolimits^* P'_H, d)\otimes\bigwedge\nolimits^*P_H\\
							&\cong (H^*(BS; \mathbb{Q})/i^*H^{>0}(BG; \mathbb{Q})\oplus \mathfrak{a})\otimes \bigwedge\nolimits^* P_H
		\end{align*}
		where $\mathfrak{a}$ is the ideal of elements of positive exterior degrees (cf. \cite[pp. 73, 83, 52]{GHV} and \cite[Theorem 2.6]{CFok}). Now suppose for the sake of contradiction that $G/S$ is not formal. Then $\mathfrak{a}\neq 0$ (\cite[Theorem 2.6]{CFok}). Any ``volume form" in $H^*(G/S; \mathbb{Q})$ is of top positive exterior degree and in particular an element of $\mathfrak{a}\otimes \bigwedge\nolimits^{\text{dim }P_H}P_H$. However, the exterior degree of $\text{ch}(ab)\text{td}(G/S)$ is 0. That is because $\text{ch}(b)\in \text{ch}(\bigwedge\nolimits^{\text{dim }P}P)=\bigwedge\nolimits^{\text{dim }P_H}P_H$ and $\text{ch}(a)\in \text{ch}(\text{im }\alpha)=H^*(BS; \mathbb{Q})/i^*H^{>0}(BG; \mathbb{Q})$ are of 0 exterior degree, and so is
		\[\text{td}(G/S)=\prod_{\alpha\in R^+}\frac{c_1(L_\alpha)}{1-e^{-c_1(L_\alpha)}}\in H^*(BS; \mathbb{Q})/i^* H^{>0}(BG; \mathbb{Q})\]
		where $L_\alpha:=G\times_S\mathfrak{g}_\alpha^{\mathbb{R}}=G\times_S\mathbb{C}_\alpha$. This way we have reached a contradiction and completed the proof.
	\end{proof}
	Theorem \ref{indexthry} we are going to prove next can be regarded as the equivariant analogue of Proposition \ref{nondegenformal}.
	\begin{proof}[Proof of Theorem \ref{indexthry}] Suppose $(G, S)$ is isotropy formal. By \cite[Theorem A]{CFok}, $G/S$ is formal. Proposition \ref{nondegenformal} implies that there exist $a\in \text{im }\alpha$ and $b\in \bigwedge\nolimits^{\text{dim }P}P$ such that $\pi_*(ab)\neq 0$. Since $P=\delta^{G/K}(\text{ker}(i^*))$ by Corollary \ref{pequaldel}, there is a natural equivariant lift $\widetilde{b}\in \bigwedge\nolimits^{\text{dim }P}_{R(S; \mathbb{Q})}\mathcal{L}_S$ of $b$. More precisely, $\widetilde{b}\in \bigwedge\nolimits^{\text{rank }G-\text{dim }S}_{R(S; \mathbb{Q})}\mathcal{L}_S$ because $\text{dim }P=\text{rank }G-\text{dim }S$ by Proposition \ref{kthrystructure}. Likewise, there is a natural equivariant lift $\widetilde{a}\in A$ of $a$ because $a$ is represented by the virtual vector bundle constructed using the line bundles $G\times_S\mathbb{C}_\mu$ associated to the representation $\mu$ of $S$, and these line bundles admit the natural $S$-equivariant structure given by the left translation on $G$. By the commutativity of $\pi_*$ and the forgetful map $f$, $f(\pi_*(\widetilde{a}\widetilde{b}))=\pi_*(ab)\neq 0$ and so $\pi_*(\widetilde{a}\widetilde{b})\notin I(S; \mathbb{Q})$ as desired.
	
	Now suppose that there exist $\widetilde{a}\in A$ and $\widetilde{b}\in\bigwedge\nolimits^{\text{rank }G-\text{dim }S}_{R(S; \mathbb{Q})}\mathcal{L}_S$ such that $\pi_*(\widetilde{a}\widetilde{b})\notin I(S; \mathbb{Q})$. Then $\pi_*(f(\widetilde{a})f(\widetilde{b}))\neq 0$. Note that $f(\widetilde{a})\in f(A)=\text{im }\alpha$ and $f(\widetilde{b})\in f(\bigwedge\nolimits^{\text{rank }G-\text{dim }S}_{R(S; \mathbb{Q})}\mathcal{L}_S)=\bigwedge\nolimits^{\text{rank }G-\text{dim }S}_\mathbb{Q}\delta^{G/K}(\text{ker}(i^*))\subseteq \bigwedge\nolimits^{\text{rank }G-\text{dim }S}_\mathbb{Q}P$. Clearly $f(\widetilde{b})\neq 0$, for otherwise $\pi_*(f(\widetilde{a})f(\widetilde{b}))=0$. So $\bigwedge\nolimits^{\text{rank }G-\text{dim }S}P\neq 0$, implying that $\text{dim }P\geq \text{rank }G-\text{dim }S$, but $\text{rank }G-\text{dim }S\leq \text{dim }P$ by Proposition \ref{kthrystructure}. It follows that $\text{dim }P=\text{rank }G-\text{dim }S$ and $\text{dim }\bigwedge\nolimits_\mathbb{Q}^{\text{rank }G-\text{dim }S}P=1$. Thus $\bigwedge\nolimits_\mathbb{Q}^{\text{rank }G-\text{dim }S}\delta^{G/K}(\text{ker}(i^*))=\bigwedge\nolimits_\mathbb{Q}^{\text{rank }G-\text{dim }S}P$ because $\bigwedge\nolimits_\mathbb{Q}^{\text{rank }G-\text{dim }S}\delta^{G/K}(\text{ker}(i^*))$ is a nonzero subspace of $\bigwedge\nolimits^{\text{rank }G-\text{dim }S}_\mathbb{Q}P$ which is one-dimensional. Then we have $P=\delta^{G/K}(\text{ker}(i^*))$ and $\text{dim }\delta^{G/K}(\text{ker}(i^*))=\text{dim }P=\text{rank }G-\text{dim }S$. By Proposition \ref{rankSam}, $(G, S)$ is isotropy formal and this completes the proof.
	\end{proof}
\begin{example}\label{isoex}
	Let $G=SU(3)$, $\displaystyle S=\left.\left\{\begin{pmatrix}t& &\\ &t^{-1}& \\ &&1\end{pmatrix}\right|t\in U(1)\right\}$, $\sigma_{\text{std}}$ the standard representation of $G$. By abuse of notation, we also use $t$ to denote the standard 1-dimensional representation of $S$
	\begin{align*}
		S&\to U(1)\\
		\begin{pmatrix}t& &\\ &t^{-1}& \\ &&1\end{pmatrix}&\mapsto t.
	\end{align*}
	Then $\displaystyle R(G)=\mathbb{Z}[\sigma_{\text{std}}, \bigwedge\nolimits^2\sigma_{\text{std}}]$ and $R(S)=\mathbb{Z}[t, t^{-1}]$. The restriction map $i^*: R(G)\to R(S)$ takes both generators $\sigma_{\text{std}}$ and $\displaystyle\bigwedge\nolimits^2\sigma_{\text{std}}$ to $1+t+t^{-1}$. So $i^*R(G; \mathbb{Q})$ is isomorphic to $\mathbb{Q}[1+t+t^{-1}]$, which is regular at $I=i^*I(G; \mathbb{Q})=(t+t^{-1}-2)$. By Theorem \ref{rep}, $(G, S)$ is an isotropy formal pair. Alternatively, as mentioned before, both restrictions $i^*\sigma_{\text{std}}$ and $i^*\bigwedge\nolimits^2\sigma_{\text{std}}$ are $1+t+t^{-1}$, which is self-dual. By Theorem \ref{circle}, we can also get the isotropy formality of $(G, S)$. Note that $\mathfrak{s}=\text{ker}(L_3)$ and $L_3\in W\cdot L_1$. By the table in Theorem \ref{corank1}, we again have that $(G, S)$ is isotropy formal.
	
	We can also use Theorem \ref{indexthry} to show the isotropy formality of $(G, S)$. Since $\text{ker}(i^*)=(\sigma_{\text{std}}-\bigwedge\nolimits^2\sigma_{\text{std}})$, $\mathcal{L}_S$ is the $R(S; \mathbb{Q})$-submodule of $K^*_S(G/S; \mathbb{Q})$ generated by $b:=\delta^{G/S}_S(\sigma_{\text{std}}(\bigwedge\nolimits^2\sigma_{\text{std}})^{-1})$. We take $a$ to be 1, represented by the equivariant trivial vector bundle over $G/S$. We will show that $\pi_*(ab)=\pi_*(b)=-1$, which is not in $I(S; \mathbb{Q})$. Let $T$ be the maximal torus $\displaystyle\left.\left\{\begin{pmatrix}t_1&&\\ &t_2&\\ && t_3\end{pmatrix}\right| t_i\in U(1), 1\leq i\leq 3, t_1t_2t_3=1\right\}$ of $SU(3)$, which contains $S$. Again by abuse of notation we use $t_i$ to denote the 1-dimensional representation which takes a diagonal matrix from $T$ to the $i$-th diagonal entry. Then the fixed point set $(G/S)^S$ of the isotropy action is $T/S\cup gT/S$, where $\displaystyle g=\begin{pmatrix}&-1&\\ 1&&\\ &&1\end{pmatrix}$. The Atiyah-Segal equivariant index formula (see \cite[Theorem 2.12]{AS}) says that 
	\[\pi_*(b)=\pi_{T/S*}\left(\frac{i_{T/S, G/S}^*(b)}{e_S(N_{T/S}^*)}\right)+\pi_{gT/S*}\left(\frac{i_{gT/S, G/S}^*(b)}{e_S(N^*_{gT/S})}\right).\]
	Here $i_{T/S, G/S}^*$ and $i_{gT/S, G/S}^*$ are restrictions from $K_S^*(G/S)$ to $K_S^*(T/S)$ and $K_S^*(gT/S)$ respectively, $N_{T/S}$ and $N_{gT/S}$ are normal bundles over $T/S$ and $gT/S$ respectively, $\pi_{T/S*}$ and $\pi_{gT/S*}$ are pushforward to $K_S^*(\text{pt})$. The short exact sequence of groups
	\[1{\longrightarrow}S\stackrel{i_{S, T}}{\longrightarrow} T\stackrel{\pi_{T, T/S}}{\longrightarrow}T/S\longrightarrow 1\]
	induces the short exact sequence of character groups
	\[0\longrightarrow \chi(T/S)\stackrel{\pi_{T, T/S}^*}{\longrightarrow} \chi(T)\stackrel{i_{S, T}^*}{\longrightarrow} \chi(S)\longrightarrow 0,\]
	where $i_{S, T}^*t_1=t$, $i_{S, T}^*t_2=t^{-1}$. The maps $i_{S, T}^*$ and $p^*$ extend by linearity to maps $i_{S, T}^*: R(T)\to R(S)$ and $\pi_{T, T/S}^*: R(T/S)\to R(T)$. We have that $R(T/S)=\mathbb{Z}[s, s^{-1}]$ with $\pi_{T, T/S}^*s=t_1t_2$. Consider the commutative diagram
	\[\xymatrix@+4pc{K_S^*(G/S)\ar[r]^{i_{T/S, G/S}^*}\ar[d]_{\pi^*_{G, G/S}}& K_S^*(T/S)\cong R(S)\otimes K^*(T/S)\ar[d]^{\pi^*_{T, T/S}}\\ K_S^*(G)\ar[r]^{i_{T, G}^*}& K_S^*(T)\cong R(S)\otimes K^*(T)}\]
	with maps induced by inclusions and projections. Here $S$ acts on $G$ and $T$ by conjugation. Then
	\begin{align*}
		i_{T}^*\circ\pi^*_{G/S}(b)&=i_T^*(\delta_S^G(\sigma_{\text{std}}-\bigwedge\nolimits^2\sigma_{\text{std}}))\\
							&=i_{S, T}^*t_1\otimes \delta(t_1)+i_{S, T}^*t_2\otimes\delta(t_2)+i_{S, T}^*t_3\otimes\delta(t_3)\\
							&-(i_{S, T}^*t_1t_2\otimes\delta(t_1+t_2)+i_{S, T}^*t_2t_3\otimes\delta(t_2+t_3)+i_{S, T}^*t_3t_1\otimes\delta(t_3+t_1))\\
							&(\text{cf. \cite[Lemma 4.19]{F}})\\
							&=t\otimes\delta(t_1)+t^{-1}\otimes\delta(t_2)+1\otimes\delta(-t_1-t_2)\\
							&-(1\otimes\delta(t_1+t_2)+t^{-1}\otimes\delta(-t_1)+t\otimes\delta(-t_2))\\
							&=(t+t^{-1}-2)\otimes\delta(t_1+t_2).
	\end{align*}
	By the commutativity of the above diagram and the injectivity of $\pi_{T/S}^*$, we have $i_{T/S}^*(b)=(t+t^{-1}-2)\otimes\delta(s)$. Similarly, $i_{gT/S}^*(b)=(t+t^{-1}-2)\otimes\delta(s)\in K_S^*(gT/S)\cong R(S)\otimes K^*(T/S)$. Next, the normal bundle $N_{T/S}$ is a trivial $T$-equivariant complex vector bundle over $T/S$ with weights of the $T$-action on the fibers being the positive roots of the Lie algebra of $G$, i.e. $t_1t_2^{-1}$, $t_2t_3^{-1}$ and $t_1t_3^{-1}$. When restricted to the $S$-action, the weights become $t^2$, $t^{-1}$ and $t$. Thus the $S$-equivariant $K$-theoretic Euler class $e_S(N_{T/S}^*)$ is $(1-t^{-2})(1-t)(1-t^{-1})$. Similarly, we find that $e_S(N_{gT/S}^*)$ is $(1-t^2)(1-t^{-1})(1-t)$. Assembling all these data to the Atiyah-Segal localization theorem yields
	\begin{align*}
		\pi_*(b)&=\pi_{T/S*}\left(\frac{(t+t^{-1}-2)\otimes\delta(s)}{(1-t^{-2})(1-t)(1-t^{-1})}\right)+\pi_{gT/S*}\left(\frac{(t+t^{-1}-2)\otimes\delta(s)}{(1-t^2)(1-t^{-1})(1-t)}\right)\\
				&=\frac{t+t^{-1}-2}{(1-t^{-2})(1-t)(1-t^{-1})}+\frac{t+t^{-1}-2}{(1-t^2)(1-t)(1-t^{-1})}\\
				&=-1
	\end{align*}
	which is not in $I(S; \mathbb{Q})$. By Theorem \ref{indexthry}, we again verify that $(G, S)$ is isotropy formal.
\end{example}

\begin{example}
	Let $G=SU(3)$ and $\displaystyle S=\left.\left\{\begin{pmatrix}t&&\\ &t&\\ &&t^{-2}\end{pmatrix}\right|t\in U(1)\right\}$. The isotropy formality of $(G, S)$ is discussed in \cite[Example 7.19]{CFok}. It is shown there that $i^*R(G; \mathbb{Q})$ is not regular at the ideal $i^*I(G; \mathbb{Q})$ and hence by Theorem \ref{rep}, $(G, S)$ is not isotropy formal. Alternatively, since $i^*\sigma_{\text{std}}=t+t+t^{-2}$ is not self-dual, we have again the non-isotropy formality by Theorem \ref{circle}. Note that $\mathfrak{s}=\text{ker}(2L_1+L_3)\notin W\cdot L_1$. By the table in Theorem \ref{corank1}, we again have that $(G, S)$ is not isotropy formal.
	
	As in Example \ref{isoex}, we will also confirm the non-isotropy formality of $(G, S)$ by using Theorem \ref{indexthry}. Let $x=\sigma_{\text{std}}$ and $\displaystyle y=\bigwedge\nolimits^2\sigma_{\text{std}}$. Then $\text{ker }i^*=4x^3+4y^3+27-x^2y^2-18xy$ (\cite[Example 7.19]{CFok}). Note that $\text{dim }(4x^3+4y^3+27)=\text{dim }(x^2y^2+18xy)=243$. By viewing these two representations as unitary linear transformations on a 243-dimensional complex vector space, we can define $b:=\delta_S^{G/S}((4x^3+4y^3+27)(x^2y^2+18xy)^{-1})$, which generates the $R(S; \mathbb{Q})$-submodule $\mathcal{L}_S$. Following the last example, we also denote the maximal torus of diagonal matrices of $SU(3)$ by $T$. Then $(G/S)^S$ is $U/S$, where $\displaystyle U=\left.\left\{\begin{pmatrix} M&\\ &\det(M)^{-1}\end{pmatrix}\right|M\in U(2)\right\}$. Note that $U/S\cong PSU(2)$, which is homeomorphic to $\mathbb{RP}^3$. Let $a\in A$. 
We will show that $\pi_*(ab)\in I(S; \mathbb{Q})$. Since $A$ is generated by the $K$-theory classes of line bundles $G\times_S\mathbb{C}_m$, $m\in\mathbb{Z}$, it suffices to show this for $a=[G\times_S\mathbb{C}_m]$. 

Let $i_{U, G}: U\to G$ be the inclusion map and $\pi_{U(1)\times SU(2), U}: U(1)\times SU(2)\to U$ be the double covering map. Then the composition $\pi^*_{U(1)\times SU(2), U}\circ i_{U, G}^*: R(G)\to R(U(1)\times SU(2))$ maps $x$ to $tz+t^{-2}$ and $y$ to $t^{-1}z+t^2$, where $t$ is the 1-dimensional standard representation of $U(1)$ and $z$ the 2-dimensional standard representation of $SU(2)$. Consider the following diagram.
	\[\xymatrix@+5pc{K_S^*(G/S)\ar[r]^{i_{U/S, G/S}^*}\ar[d]_{\pi_{G, G/S}^*}&K_S^*(U/S)\ar[r]^{\pi^*_{SU(2), U/S}}\ar[d]_{\pi_{U, U/S}^*}& K_S^*(SU(2))\ar[d]^{\pi_{U(1)\times SU(2), SU(2)}^*}\\ K_S^*(G)\ar[r]^{i_{U, G}^*}&K_S^*(U)\ar[r]^{\pi^*_{U(1)\times SU(2), U(1)}}& K_S^*(U(1)\times SU(2))}\]

Here $S$ acts on $U$, $U/S$, $U(1)\times SU(2)$ and $SU(2)$ trivially. The vertical maps are induced by the quotient maps and the horizontal ones by inclusion and double covering maps. The diagram can be easily seen to be commutative. We have
\begin{align*}
	\pi^*_{U(1)\times SU(2), U}\circ i_{U, G}^*\circ\pi_{G, G/S}^*(b)&=\pi_{U(1)\times SU(2), U}^*\circ i_{U, G}^*(\delta_S^G(4x^3+4y^3+27-x^2-y^2-18xy))\\
													&=\delta_S^{U(1)\times SU(2)}(-(z^2-4)(t^3+t^{-3}-z)^2)\\
													&=\text{dim}(-(z^2-4))\delta_S^{U(1)\times SU(2)}((t^3+t^{-3}-z)^2)+\\
													&(t^3+t^{-3}-\text{dim}(z))^2\delta_S^{U(1)\times SU(2)}(4-z^2)\\
													&=-(t^3+t^{-3}-2)^2\delta_S^{U(1)\times SU(2)}(z^2)\\
													&=-(t^3+t^{-3}-2)^2\cdot 2\text{dim}(z)\delta_S^{U(1)\times SU(2)}(z)\\
													&=-4(t^3+t^{-3}-2)^2\delta_S^{U(1)\times SU(2)}(z).
\end{align*}
By commutativity and injectivity of the four maps on the degree $-1$ pieces of the $K$-theory group in the right square of the diagram, we have
\[i_{U/S, G/S}^*(b)=-2(t^3+t^{-3}-2)^2r,\]
where $r$ is the free generator of $K^{-1}_S(U/S)$ as a $R(S)$-module which satisfies $\pi_{SU(2), U/S}^*(r)=2\delta_S^{SU(2)}(z)$ and $\pi_{U/S*}(r)=1$. Note that $TG/S\cong G\times_S(\mathfrak{g}/\mathfrak{s})$ and $TU/S\cong U\times_S(\mathfrak{u}/\mathfrak{s})$, where $S$ acts on $\mathfrak{g}$ and $\mathfrak{u}$ by the adjoint action. Since $\mathfrak{g}\cong\mathfrak{t}\oplus\bigoplus_{1\leq i<j\leq 3}\mathfrak{g}_{L_i-L_j}$ and $\mathfrak{u}\cong\mathfrak{t}\oplus\mathfrak{g}_{L_1-L_2}$, where $\mathfrak{t}$ is the Lie algebra of the diagonal subgroup and $\mathfrak{g}_{L_i-L_j}$ the root space corresponding to the root $L_i-L_j$, we have that the normal bundle $N_{U/S}$ of $U/S$ in $G/S$ is $U\times_S(\mathfrak{g}_{L_1-L_3}\oplus\mathfrak{g}_{L_2-L_3})$. On the other hand, note that $U\times_S (\mathbb{C}_m\oplus\mathbb{C}_m)$ is equivariantly isomorphic to the trivial vector bundle $U/S\times(\mathbb{C}_m\oplus\mathbb{C}_m)$. That is because $\mathbb{C}_m\oplus\mathbb{C}_m$ is the restriction of some $U$-representation to $S$. Then the normal bundle $N_{U/S}$ is equivariantly isomorphic to $ U/S\times (\mathfrak{g}_{L_1-L_3}\oplus\mathfrak{g}_{L_2-L_3})\cong U/S\times(\mathbb{C}_3\oplus\mathbb{C}_3)$ as both $\mathfrak{g}_{L_1-L_3}$ and $\mathfrak{g}_{L_2-L_3}$ both restrict to the $S$-representation of weight being 3. Thus the $S$-equivariant $K$-theory class of $N_{U/S}$ is $t(t^{-2})^{-1}+t(t^{-2})^{-1}=2t^3$, and the Euler class $e_S(N^*_{U/S})$ is $(1-t^{-3})^2$. By the Atiyah-Segal equivariant index formula,
	\begin{align*}
		\pi_*(ab)&=\pi_{U/S*}\left(\frac{i_{U/S, G/S}^*(ab)}{e_S(N_{U/S}^*)}\right)\\
				&=\pi_{U/S*}\left(\frac{i_{U/S, G/S}^*(a)(-2(t^3+t^{-3}-2)^2r)}{(1-t^{-3})^2}\right)\\
				&=\pi_{U/S*}(i_{U/S, G/S}^*(a)(-2(t^3-1)^2)r)\\
				&=\pi_{U/S*}(i_{U/S, G/S}^*(2a)(-(t^3-1)^2)r)\\
				&=\pi_{U/S*}([U\times_S\mathbb{C}_m^{\oplus 2}]\cdot(-(t^3-1)^2)r)\\
				&=\pi_{U/S*}(-2t^m(t^3-1)^2r)\\
				&=-2t^m(t^3-1)^2\pi_{U/S*}(r)\ (\text{Push-pull formula})\\
				&=-2t^m(t^3-1)^2
	\end{align*}
	which is in $I(S; \mathbb{Q})$. By Theorem \ref{indexthry}, we again verify that $(G, S)$ is not isotropy formal.
\end{example}

	\noindent\footnotesize{\textsc 
	\\
	Department of Applied Mathematics,\\
	School of Mathematics and Physics,\\
	Xi'an Jiaotong-Liverpool University,\\
111 Ren’ai Road, Suzhou Industrial
Park,\\Suzhou, Jiangsu Province 215123, China}\\
\\
\textsc{E-mail}: \texttt{ChiKwong.Fok@xjtlu.edu.cn}\\
\end{document}